\documentclass[final]{siamltex} 
\usepackage[utf8]{inputenc}
\usepackage{amsmath,amssymb,amsfonts,etoolbox}
\usepackage[hidelinks]{hyperref}
\hypersetup{colorlinks=true,citecolor=blue,linkcolor=blue,
            filecolor=blue,urlcolor=blue}
\usepackage{graphicx}
\usepackage{cite}
\usepackage[font={small}]{caption}
\usepackage{url}
\usepackage{booktabs}
\usepackage{framed}
\usepackage{mathtools}
\usepackage{pgf,tikz}
\usepackage{mathrsfs}
\usetikzlibrary{arrows}

\newcommand{\bb}[1]{\textcolor{blue}{#1}}

\newcommand{\ignore}[1]{}

\DeclareMathOperator*{\minimize}{minimize}

\def\Zt{\widetilde Z}

\newcommand\T{\rule{0pt}{3ex}}       % Top strut
\newcommand\B{\rule[-1.5ex]{0pt}{0pt}} % Bottom strut
\AtBeginEnvironment{bmatrix}{\setlength{\arraycolsep}{3pt}}

\begin{document}

\title{
Rational minimax approximation via adaptive barycentric
representations
}

\author{Silviu-Ioan Filip\thanks{Univ Rennes, Inria, CNRS, IRISA, F-35000 Rennes, France
		(silviu.filip@inria.fr).}
\and  Yuji Nakatsukasa\thanks{
National Institute of Informatics, 2-1-2 Hitotsubashi, Chiyoda-ku, Tokyo 101-8430, Japan.  (nakatsukasa@nii.ac.jp)}
	%Version of \today.
\and Lloyd N. Trefethen\thanks{
	Mathematical Institute,	University of Oxford,
	Oxford, OX2 6GG, UK
	(trefethen@maths.ox.ac.uk).
	SF and LNT were supported by the European Research Council under the
	European Union’s Seventh Framework Programme (FP7/2007–2013)/ERC grant agreement 291068.
	The views expressed in this article are not those of the ERC or the European Commission, and the
	European Union is not liable for any use that may be made of the information contained here. YN was supported by Japan Society for the Promotion of Science as an Overseas Research Fellow. }
 %\footnotemark[2]
\and Bernhard Beckermann\thanks{
Laboratoire Paul Painleve UMR 8524, Dept.\ Math\'ematiques, Univ.\ Lille, F-59655 Villeneuve d'Ascq CEDEX, France
%Laboratoire Paul Painleve UMR 8524, Dept.\ Math\'ematiques, Univ.\ Lille, F-59655 Villeneuve d'Ascq CEDEX, France
 (bernhard.beckermann@univ-lille1.fr). Supported in part by the Labex CEMPI (ANR-11-LABX-0007-01).
}
}

\maketitle

\begin{abstract}
Computing rational minimax approximations can be very challenging
when there are singularities on or near the interval of approximation
--- precisely the case where rational functions outperform polynomials
by a landslide.  We show that far more robust algorithms than
previously available can be developed by making use of rational
barycentric representations whose support points are chosen in an
adaptive fashion as the approximant is computed.  Three variants of
this  barycentric strategy are all shown to be powerful:
(1) a classical Remez algorithm,
(2) a ``AAA-Lawson'' method of iteratively reweighted least-squares, and (3) a differential correction
algorithm.  Our preferred combination, implemented in the Chebfun
MINIMAX code, is to use (2) in an initial phase and then switch
to (1) for generically quadratic convergence.  By such methods we can
calculate approximations up to type $(80,80)$ of $|x|$ on $[-1,1]$
in standard 16-digit floating point arithmetic, a problem for which
Varga, Ruttan, and Carpenter required 200-digit extended precision.

\end{abstract}

\begin{keywords}barycentric formula, rational minimax approximation, Remez algorithm, differential correction algorithm, AAA algorithm, Lawson algorithm\end{keywords}

\begin{AMS}
41A20, 65D15
\end{AMS}

\pagestyle{myheadings} 
\thispagestyle{plain} 
\markboth{
FILIP, 
 NAKATSUKASA, 
TREFETHEN AND
BECKERMANN}
{\uppercase{
RATIONAL MINIMAX APPROXIMATION
}}

\section{Introduction}\label{sec:intro}
The problem we are interested in is that of approximating functions $f\in\mathcal{C}([a,b])$ using type $(m,n)$ rational approximations with real coefficients, in the $L^\infty$ setting. The set of feasible approximations is
\begin{equation}  \label{eq:ratquotient}
\mathcal{R}_{m,n}=\left\lbrace \frac{p}{q}:\  p\in\mathbb{R}_m[x],\ q\in\mathbb{R}_n[x]\right\rbrace.  
\end{equation}
Given $f$ and prescribed nonnegative integers $m,n$, the goal is to compute
\begin{equation}  \label{eq:mainproblem}
\min_{r\in\mathcal{R}_{m,n}}\|f-r\|_{\infty},
\end{equation}
where $\|\cdot\|_\infty$ denotes the infinity norm over $[a,b]$, i.e.,~$\|f-r\|_{\infty}=\max_{x\in[a,b]}|f(x)-r(x)|$.
The minimizer of~\eqref{eq:mainproblem} is known to exist and to be unique~\cite[Ch. 24]{trefethen2013approximation}.

Let the~\emph{minimax} (or~\emph{best}) approximation be written $r^*=p/q\in\mathcal{R}_{m,n}$, where $p$ and $q$ have no common factors. 
The number $d=\min\left\lbrace m-\deg p, m-\deg q\right\rbrace$ is called the \emph{defect} of $r^*$.
It is known that there exists a so-called~\emph{alternant} (or~\emph{reference}) set consisting of ordered nodes $a\leqslant x_0<x_1<\cdots<x_{m+n+1-d}\leqslant b$, where $f-r^*$ takes its global extremum over $[a,b]$ with alternating signs. In other words, we have the  beautiful~\emph{equioscillation} property~\cite[Theorem~24.1]{trefethen2013approximation}
\begin{equation}\label{eq:alternori}
	f(x_\ell)-r^{*}(x_\ell)=(-1)^{\ell+1}\lambda, \qquad \ell=0,\ldots, m+n+1-d,
\end{equation}
where $|\lambda|=\|f-r^*\|_{\infty}$. Minimax approximations with $d>0$ are called~\emph{degenerate}, and they can cause problems for computation. 
Accordingly, unless otherwise stated, we make the assumption that $d=0$ for~\eqref{eq:mainproblem}. In practice, degeneracy most often arises due to symmetries in approximating even or odd functions, and we check for these cases explicitly to make sure they are treated properly. Other degeneracies can usually be detected by examining in succession the set of best approximations of types $(m-k,n-k), (m-k+1,n-k+1),\ldots,(m,n)$ with $k=\min \left\lbrace m,n\right\rbrace$~\cite[p.~161]{braess2012nonlinear}.

In the approximation theory literature~\cite{cheney1982introduction,powell1981approximation,watson1980approximation,braess2012nonlinear,meinardus1967approximation}, two algorithms are usually considered for the numerical solution of~\eqref{eq:mainproblem}, the~\emph{rational Remez} and~\emph{differential correction} (DC) algorithms. The various challenges that are inherent in rational approximations can, more often than not, make the use of such methods difficult. 
Finding the best polynomial approximation, by contrast, can usually be done robustly by a standard implementation of the linear version of the Remez algorithm~\cite{PachTrefethen09}.
This might explain why the current software landscape for minimax rational approximations is rather barren. Nevertheless, implementations of the rational Remez algorithm are available in some mathematical software packages: the Mathematica~\texttt{MiniMaxApproximation} function, the Maple~\texttt{numapprox[minimax]} routine and the MATLAB Chebfun~\cite{Driscoll2014}~\texttt{remez} code. The Boost C++ libraries~\cite{boostsite} also contain an implementation. 

Over the years, the applications that have benefited most from minimax rational approximations come from recursive filter design in signal processing~\cite{deczky1974equiripple,brophy1975synthesis} and the representation of special functions~\cite{cody1975funpack,cody1993specfun}. Apart from such practical motivations, we believe it worthwhile to pursue robust numerical methods for computing these approximations because of their fundamental importance to approximation theory.
A new development
of this kind has already resulted from the algorithms described here:
the discovery that type $(k,k)$ rational approximations to $x^n$,
for $n\gg k$, converge geometrically at the rate $O(9.28903\cdots ^{-k})$~\cite{nakatsukasa2018rational}. 
%For example, our code \texttt{minimax} was used to compute the best low-degree rational approximant to $x^n$ in~\cite{nakatsukasa2018rational}, revealing the convergence is independent of $n$ in the limit $n\rightarrow \infty$.

In this paper we present elements that greatly improve the numerical robustness of algorithms for computing best rational approximations. The key idea is the use of barycentric representations with 
adaptively chosen basis functions, which can overcome the numerical difficulties frequently encountered when $f$ has nonsmooth points. 
For instance, when trying to approximate $f(x)=|x|$ on $[-1,1]$ using standard IEEE double precision arithmetic in MATLAB, our barycentric Remez algorithm can compute rational approximants of type up to $(82,82)$---higher than that obtained by Varga, Ruttan and Carpenter in~\cite{VargaEtAl93} using $200$-digit arithmetic\footnote{Chebfun's previous {\tt remez} command (until version 5.6.0 in December 2016) could only go up to type $(8,8)$.}.

A similar Remez iteration using the barycentric representation was described by Ioni\textcommabelow{t}\u{a}~\cite[Sec.~3.2.3]{ionita2013lagrange} in his PhD thesis. We adopt the same set of \emph{support points} (see Section~\ref{sec:qrcond}), 
%our algorithm formulation is different and 
% as those suggested by Ionita 
%the algorithm formulation is different, we use the same support points as those suggested by Ionita. 
and our analysis justifies its choice: we prove its optimality in a certain sense. A difference from Ioni\textcommabelow{t}\u{a}'s
 treatment is that we reduce the core computational task to a symmetric eigenvalue problem, rather than a generalized eigenproblem as in~\cite{ionita2013lagrange}. The bigger difference is that Ioni\textcommabelow{t}\u{a} treated just the core iteration for approximations of type $(n,n)$, whereas we generalize the approach to type $(m,n)$ and include the initialization strategies that are crucial for making the entire procedure into a fully practical algorithm.
%Ioni\textcommabelow{t}\u{a}

%\rrr{This idea was used with great success in the AAA algorithm~\cite{aaapreprint} for rational approximation, and this work is largely an outgrowth of AAA. Here we devise a rational Remez algorithm in barycentric form, implemented in Chebfun. Experiments show significant improvements over previous codes, especially for nonsmooth functions. For instance, when trying to approximate $f(x)=|x|$ on $[-1,1]$ using standard IEEE double precision arithmetic in MATLAB, our barycentric Remez algorithm can compute rational approximants of type up to $(82,82)$---higher than that obtained by Varga, Ruttan and Carpenter in~\cite{VargaEtAl93} using $200$ digit arithmetic\footnote{Chebfun's previous {\tt remez} command (until version 5.6.0) could only go up to type $(8,8)$.}. We also examine the use of barycentric representation for the differential correction algorithm, and find that it exhibits a similar improvement in robustness.}

This work is motivated by the recent~\emph{AAA algorithm}~\cite{aaapreprint} for rational approximation, 
which uses adaptive barycentric representations with great success. A large part of the text is focused on introducing a robust version of the rational Remez algorithm, followed by a discussion of two other methods for discrete $\ell_\infty$ rational approximation: the AAA-Lawson algorithm (efficient at least in the early stages, but non-robust) and the DC algorithm (robust, but not very efficient). We shall see how all three algorithms benefit from an adaptive barycentric basis. In practice, we advocate using the Remez algorithm, mainly for its convergence properties (usually quadratic~\cite{curtisosborne1966quadratic}, unlike AAA-Lawson, which converges linearly at best), practical speed (an eigenvalue-based Remez implementation is usually much faster than a linear programming-based DC method), and its ability to work with the interval $[a,b]$ directly rather than requiring a discretization (unlike both AAA-Lawson and DC). AAA-Lawson is used mainly as an efficient approach to initialize the Remez algorithm.

%\rr{This paper also presents two other algorithms for~\eqref{eq:mainproblem}: the (new) AAA-Lawson algorithm, and the (classical) differential-correction (DC) algorithm. We shall see that all three algorithms benefit significantly from the barycentric representation with an adaptive choice of support points. 
%In practice, we advocate the use of Remez for its convergence (usually quadratic; unlike AAA-Lawson), 
%practical speed (eigenvalue-based Remez is usually much faster than linear programming-based DC)
%and lack of need for discretizing the interval (unlike both AAA-Lawson and DC). 
%Here we use AAA-Lawson as an efficient means to obtain a good initialization for the Remez algorithm. 
%}

The paper is organized as follows. In Section~\ref{sec:barybasic} 
we review the barycentric representation for rational functions. 
%where we develop the barycentric rational Remez algorithm. 
%, discuss its properties and the reason the representation is superior to the standard ``quotient of polynomials'' form. 
Sections~\ref{sec:remezbasics} to~\ref{sec:nextrefs} are the core of the paper; here we develop the barycentric rational Remez algorithm with adaptive basis functions.
Numerical experiments are presented in Section~\ref{sec:numres}. %\rrrr{to illustrate the computation of a variety of best rational approximants, made possible by our barycentric rational Remez}
We describe the AAA-Lawson algorithm in Section~\ref{sec:iniAAA}, and  
in Section~\ref{sec:dc} we briefly present the barycentric version of the differential correction algorithm. %, \rrr{which exhibits similar improvement in numerical robustness}{
Section~\ref{sec:conclude} presents a flow chart of {\tt minimax} and an example of how to compute a best approximation in Chebfun.

\section{Barycentric rational functions}\label{sec:barybasic}
All of our methods are made possible by a barycentric representation of $r$, in which both the numerator and denominator are given as partial fraction expansions. Specifically, we consider
\begin{equation}
\label{eq:baryr}
r(z)=\frac{N(z)}{D(z)}=\sum_{k=0}^n\dfrac{\alpha_k}{z-t_k}\bigg/\sum_{k=0}^n\dfrac{\beta_k}{z-t_k},   
%\frac{\sum_{k=0}^n\dfrac{\alpha_k}{z-t_k}}{\sum_{k=0}^n\dfrac{\beta_k}{z-t_k}},   
\end{equation}
where $n\in\mathbb{N}$, 
$\alpha_0,\ldots,\alpha_n$ and $\beta_0,\ldots,\beta_n$ are sets of real coefficients and $t_0,\ldots,t_n$ is a set of distinct real \emph{support points}. The names $N$ and $D$ stand for ``numerator'' and ``denominator''. 

If we denote by $\omega_t$ the \emph{node polynomial} associated with $t_0,\ldots,t_n$,
\[
\omega_t(z)=\prod_{k=0}^{n}(z-t_k),
\]
then $p(z)=\omega_t(z)N(z)$ and $q(z)=\omega_t(z)D(z)$ are both polynomials in $\mathbb{R}_n[x]$. We thus get $r(z)=p(z)/q(z)$, meaning that $r$ is a type $(n,n)$ rational function. (This is not necessarily sharp; $r$ may also be of type $(\mu,\nu)$ with $\mu<n$ and/or $\nu<n$.) At each point $t_k$ with nonzero $\alpha_k$ or $\beta_k$, formula~\eqref{eq:baryr} is undefined, but this is a removable singularity with $\lim_{z\rightarrow t_k}r(z)=\alpha_k/\beta_k$ (or a simple pole in the case $\alpha_k\neq 0, \beta_k=0$), meaning $r$ is a \emph{rational interpolant} to the values $\left\lbrace\alpha_k/\beta_k\right\rbrace$ at the support points $\left\lbrace t_k\right\rbrace$.

Much of the literature on barycentric representations exploits this interpolatory property~\cite{schneider1986some,berrut1988rational,berrut2004barycentric,berrut2005recent,floater2007barycentric,brezinski2013pade} by taking $\alpha_k=f(t_k)\beta_k$, so that $r$ is an interpolant to some given function values $f(t_0),\ldots,f(t_n)$ at the support points. In this case 
\begin{equation}  \label{eq:rzbary}
r(z) =\sum_{k=0}^n\dfrac{f(t_k)\beta_k}{z-t_k}\bigg/\sum_{k=0}^n\dfrac{\beta_k}{z-t_k},  
%\frac{\sum_{k=0}^n\dfrac{f(t_k)\beta_k}{z-t_k}}{\sum_{k=0}^n\dfrac{\beta_k}{z-t_k}},  
\end{equation}
with the coefficients $\left\lbrace\beta_k\right\rbrace$ commonly known as \emph{barycentric weights}; we have $r(t_k)=f(t_k)$  as long as $\beta_k\neq 0$. While such a property is useful and convenient when we want to compute good approximations to $f$ (see in particular the AAA algorithm), for a best rational approximation $r^*$ we do not know a priori where $r^*$ will intersect $f$, so enforcing interpolation is not always an option. (We use interpolation for Remez but not for AAA-Lawson or DC.) Formula~\eqref{eq:baryr}, on the other hand, has  $2n+1$ degrees of freedom and can be used to represent any rational function of type $(n,n)$ by appropriately choosing $\left\lbrace\alpha_k\right\rbrace$ and $\left\lbrace\beta_k\right\rbrace$~\cite[Theorem~2.1]{aaapreprint}. We remark that variants of~\eqref{eq:baryr} also form the basis for the popular vector fitting~\cite{gustavsen1999rational,gustavsen2006improving} method used to match frequency response measurements of dynamical systems. A crucial difference is that the support points $\{t_k\}$ in vector fitting are selected to approximate poles of $f$, whereas, as we shall describe in detail, we choose them so that our representation uses a numerically stable basis.

\subsection{Representing rational functions of nondiagonal type}\label{subsec:nondiag}
%The barycentric rational function~\eqref{eq:baryr} is generically of type $(n,n)$; for example with randomly chosen coefficients $\alpha_k,\beta_k$, $r$ is of exact type $(n,n)$. On the other hand, 
Functions $r$ expressed in the barycentric form~\eqref{eq:baryr} range precisely over the set of all rational functions of (not necessarily exact) type $(n,n)$. 
%\rrrr{For a generic choice of $\alpha_k,\beta_k$, $r$ of exact type $(n,n)$.} 
When one requires rational functions of type $(m,n)$
 with $m\neq n$, additional steps are needed to enforce the type. 

The approach we have followed,
which we shall now describe, is a linear algebraic one based on previous work by Berrut and Mittelmann~\cite{berrut1997matrices}, where we
make use of Vandermonde matrices to impose certain conditions that
limit the numerator or denominator degree.  An alternative might be to avoid such matrices and constrain the barycentric representation
more directly to have a certain number of poles or zeros at $z=\infty$.
This is a matter for future research.

To examine the situation, we first suppose $m<n$ and convert $r$ into the conventional polynomial quotient representation 
%of~\eqref{eq:baryr}
%the function 
\begin{equation}
  \label{eq:linbarypoly}
r(z)=
\frac{\omega_t(z)N(z)}{\omega_t(z) D(z)}
=\frac{\prod_{k=0}^n(z-t_k)\sum_{k=0}^n\dfrac{\alpha_k}{z-t_k}}{\prod_{k=0}^n(z-t_k)\sum_{k=0}^n\dfrac{\beta_k}{z-t_k}}=:\frac{p(z)}{q(z)}.   
\end{equation} 
The numerator $p$ is a polynomial
%$q(x)=\prod_{k=0}^n(x-t_k)\sum_{k=0}^n\dfrac{\beta_k}{x-t_k}$
 of degree at most $n$. 
Further, it can be seen (either via direct computation or from~\cite[eq.~(1)]{berrut1997matrices}) that $p$ is of degree $m\ (<n)$ if and only if the vector $\alpha=[\alpha_{0},\ldots,\alpha_n]^T$ lies in a subspace 
spanned by the null space of the (transposed) Vandermonde matrix
%of a Vandermonde matrix, namely the null complex conjugate of 
\begin{equation}  \label{eq:Vm}
V_m =   \begin{bmatrix}
1     &    1    &  \cdots & 1\\
t_0 &    t_1 & \cdots &  t_{n}\\
\vdots   &     \vdots   &  & \vdots \\
t_0^{n-1-m}&t_1^{n-1-m} & \cdots & t_{n}^{n-1-m}
\end{bmatrix}. 
\end{equation}
That is, to enforce $r\in\mathcal{R}_{m,n}$ with $m<n$,
we 
require $\alpha\in\mbox{span}(P_m)$, where 
$P_m\in\mathbb{R}^{(n+1)\times (m+1)}$ has orthonormal columns, 
obtained by taking the full QR factorization 
$V_m^T=
\begin{bmatrix}
P_m^\perp& P_m  
\end{bmatrix}\begin{bmatrix}
R_m\\
0  
\end{bmatrix}
$, 
where $P_m^\perp\in\mathbb{R}^{(n+1)\times (n-m)}$, 
$R_m\in\mathbb{R}^{(n-m)\times (n-m)}$. 
Note that $R_m$ is nonsingular if the support points $\{t_k\}$ are distinct.

Similarly, for $m>n$, we need to 
take $m+1$ terms in~\eqref{eq:baryr}, that is, 
$r(z)=\sum_{k=0}^{m}\alpha_k(z-t_k)^{-1}\big/\sum_{k=0}^m\beta_k(z-t_k)^{-1}$, and 
force 
$\beta\in \mbox{span}(P_n)$, where
% is chosen such that % (YN: maybe add more). 
%When $m<n$, 
$\mbox{span}(P_n)$ is the null space of the matrix 
%, with $\alpha$ 
\begin{equation}\label{eq:Vn}
V_n=  \begin{bmatrix}
1     &    1    &  \cdots & 1\\
t_0 &    t_1 & \cdots &  t_{m}\\
\vdots   &     \vdots   &  & \vdots \\
t_0^{m-1-n}&t_1^{m-1-n} & \cdots & t_{m}^{m-1-n}
\end{bmatrix}, 
\end{equation}
obtained by the QR factorization 
$V_n^T=
\begin{bmatrix}
P_n^\perp& P_n  
\end{bmatrix}
\begin{bmatrix}
R_n\\
0  
\end{bmatrix}
$, where $P_n^\perp\in\mathbb{R}^{(m+1)\times (m-n)}$, 
$R_n\in\mathbb{R}^{(m-n)\times (m-n)}$.

In Section~\ref{sec:nondiag} we describe how to use the matrices $P_m,P_n$ in specific situations. 
Since these matrices are obtained via $V_m,V_n$ in~\eqref{eq:Vm}--\eqref{eq:Vn} and real-valued Vandermonde matrices are usually highly ill-conditioned~\cite{beckermann2000condition,beckermann2017singular,pan2016bad}, care is needed when computing their null spaces, as extracting the orthogonal factors in QR (or SVD) is susceptible to numerical errors. 
Berrut and Mittelmann~\cite{berrut1997matrices} suggest a careful elimination process to remedy this (for a slightly different problem). 
Here, in view of the Krylov-type structure of the matrices $V_m^T$ and $V_n^T$, 
we propose the following simpler approach, based on an Arnoldi-style orthogonalization: 
\medskip
\begin{enumerate}
	\item Let $Q=[1,\ldots,1]^T$ when $m>n$, and 
	$Q = [f(t_0),\ldots,f(t_n)]^T$ when $m<n$, and normalize to have Euclidean norm 1. 
	\item Let $q$ be the last column of $Q$. 
	Take the projection of $\mbox{diag}(t_0,\ldots,t_{\max(m,n)})q$ 	onto the orthogonal complement of $Q$, normalize, and append it to the right of $Q$.
 Repeat this $|m-n|$ times to obtain $Q\in\mathbb{C}^{(\max(m,n)+1)\times (|m-n|)}$. In MATLAB, this is 
	{\tt q = Q(:,end); q = diag(t)*q; 
for i = 1:size(Q,2), q = q-Q(:,i)*(Q(:,i)'*q); end,
 q = q/norm(q); Q = [Q,q];}. 
%{A simple (perhaps better) alternative is to use CGS2, in which {\tt q = q-Q*(Q'*q); q = q-Q*(Q'*q);} replaces the for loop.}
	\item Take the orthogonal complement $Q^\perp$ of $Q$ via computing the QR factorization of $Q$. $Q^\perp$ is the desired matrix, $P_m$ or $P_n$.
	%Then find $\min_{\|w\|_2, w\in\mbox{span}(Q^{\perp})}\|Lw\|_2$ via $\min_{\|w\|_2}\|LQ^{\perp}w\|_2$. 
	%Then $w$ shall lie in $\mbox{span}(Q^\perp)$; this is achieved by computing the smallest singular vector of $LQ^\perp$. 
\end{enumerate}
\medskip
Note that the matrix $Q$ in the final step is well conditioned ($\kappa_2(Q)=1$ in exact arithmetic), so the final QR factorization is a stable computation. 
%(YN: This process gives accuracy similar to Berrut and Mittelmann's algorithm for the polynomial case $n=0$). 

%Here $P_m\in\mathbb{R}^{(n+1)\times(m+1)},P_n\in\mathbb{R}^{(m+1)\times(n+1)}$ are matrices of orthonormal columns. 

\subsection{Why does the barycentric representation help?}\label{sec:whybary}
The choice of the support points $\left\lbrace t_k\right\rbrace$ is very important numerically, and indeed it is the flexibility of where to place these points that is the source of the power of barycentric representations. If the points are well chosen, the basis functions $1/(x-t_k)$ lead to a representation of $r$ that is much better conditioned (often exponentially better) than the conventional representation as a ratio of polynomials. 
We motivate and explain our adaptive choice of $\left\lbrace t_k\right\rbrace$ for the Remez algorithm in Sections~\ref{sec:qrcond} and~\ref{sec:supportpts}. The analogous choices for AAA-Lawson and DC are discussed in Sections~\ref{sec:aaasupport} and~\ref{sec:dcsupport}.

To understand why a barycentric representation is preferable
for rational approximation, we first consider the standard
quotient representation $p/q$.  It is well known that a polynomial
will vary in size by exponentially large factors over
an interval unless its roots are suitably distributed
(approximating a minimal-energy configuration).  If $p/q$ is
a rational approximation, however, the zeros of $p$ and $q$ will
be positioned by approximation considerations, and if $f$ has
singularities or near-singularities they will be clustered
near those points.  In the clustering region, $p$ and $q$ will be
exponentially smaller than in other parts of the interval and
will lose much or all of their relative accuracy.  Since the
quotient $p/q$ depends on that relative accuracy, its accuracy
too will be lost.

\begin{figure}[htbp]
	\centering
	\includegraphics[width=0.6\textwidth]{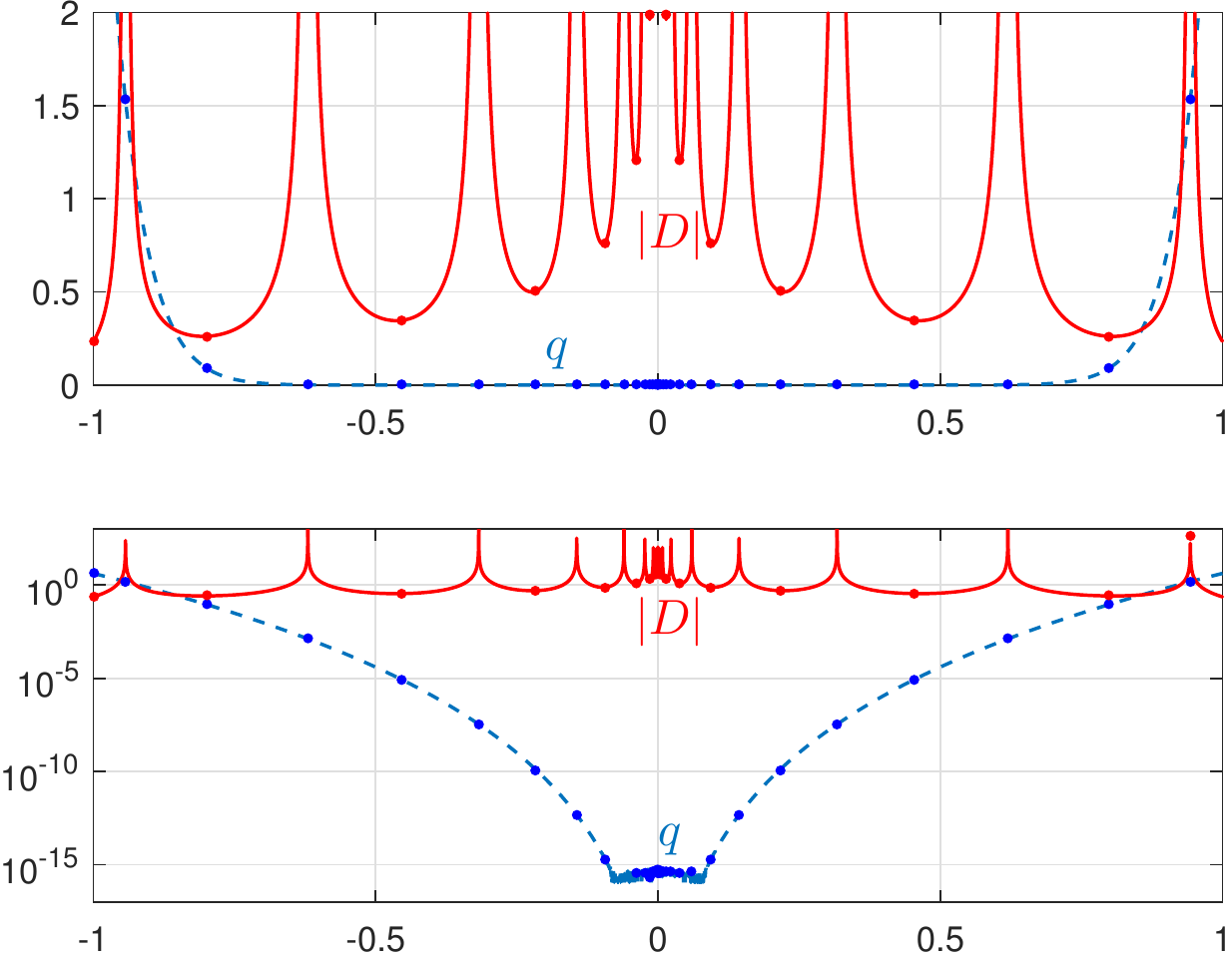}  
	\caption{
		Linear (top) and semilogy (bottom) plots of $q$ and $|D|$ in $r^{*}=p/q=N/D$, 
		the best rational approximation for $|x|$ of type $(m,n)=(20,20)$.
		Here $p,q$ are the polynomials in the classical quotient representation~\eqref{eq:ratquotient}, and $D$ is the denominator in the barycentric representation~\eqref{eq:baryr}. The dots are the equioscillation points $\{x_\ell\}$, while the set of support points $\{t_k\}$ consists of every other point in $\{x_\ell\}$.
	}
	\label{fig:plotq}
	
\end{figure}

A barycentric quotient $N/D$, by contrast, is composed of terms
that  vary in size just algebraically across the interval, not
exponentially, so this effect does not arise.  If the support
points are suitably clustered, $N$ and $D$ may have approximately
uniform size across the interval (away from their poles, which
cancel in the quotient), as illustrated in Figure~\ref{fig:plotq}.

\subsection{Numerical stability of evaluation}
Regarding the evaluation of $r$ in the barycentric representation, 
Higham's analysis in~\cite[p.~551]{higham2004numerical} 
(presented for barycentric polynomial interpolation, but equally valid for~\eqref{eq:baryr}) shows that evaluating $r(x)$ is backward stable in the sense that the computed value $\widehat r(x)$ satisfies 
\begin{equation}  \label{eq:backsta}
\widehat r(x) = \sum_{k=0}^n\dfrac{\alpha_k(1+\epsilon_{\alpha_k})}{x-t_k}\bigg/\sum_{k=0}^n\dfrac{\beta_k(1+\epsilon_{\beta_k})}{x-t_k},  
\end{equation}
where $\epsilon_{\alpha_k},\epsilon_{\beta_k}$ denote quantities of size $O(u)$, or more precisely, bounded by $(1+u)^{3n+4}$. 
In other words, $\widehat r(x)$ is an exact evaluation of~\eqref{eq:baryr} for slightly perturbed $\{\alpha_k\},\{\beta_k\}$. Note that when $r$ represents a polynomial (as assumed in~\cite{higham2004numerical}),~\eqref{eq:backsta} does not imply backward stability. 
However, as a rational function for which we allow for backward errors in the denominator,~\eqref{eq:backsta} does imply backward stability.

For the forward error, we can adapt the analysis of~\cite[Proposition 2.4.3]{celis2008practical}. Assume that the computed coefficients $\widehat{\alpha},\widehat{\beta}$ are obtained through a backward stable process, 
\[
\widehat{\alpha}_k = \alpha_k(1+\delta_{\alpha_k}),\ \delta_{\alpha_k}=O(\kappa_\alpha u),\ \widehat{\beta}_k = \beta_k(1+\delta_{\beta_k}), \ \delta_{\beta_k}=O(\kappa_\beta u), \quad k = 0,\ldots,n,
\]
where $\kappa_\alpha$ and $\kappa_\beta$ are condition numbers associated with the matrices used to determine $\widehat{\alpha}$ and $\widehat{\beta}$. Then, if $x$ (the evaluation point) and $\{t_k\}$ are considered to be floating point numbers,  we have
\begin{lemma}\label{lemma:forwardrelerror}
	The relative forward error for the computed value $\widehat{r}(x)$ of~\eqref{eq:baryr} satisfies
	\begin{equation}          \label{eq:evalerr}
	\left|\frac{r(x)-\widehat{r}(x)}{r(x)}\right|\leq u(n+3+O(\kappa_\alpha))\frac{\sum_{k=0}^{n}\left|\frac{\alpha_k}{x-t_k}\right|}{\left|\sum_{k=0}^{n}\frac{\alpha_k}{x-t_k}\right|}+u(n+2+O(\kappa_\beta))\frac{\sum_{k=0}^{n}\left|\frac{\beta_k}{x-t_k}\right|}{\left|\sum_{k=0}^{n}\frac{\beta_k}{x-t_k}\right|}+O(u^2).          
	\end{equation}
\end{lemma}
\begin{proof}
	This follows from~\cite[Prop.~2.4.3]{celis2008practical}.
\end{proof}

If the functions $|D(x)|$ and $|N(x)|$ appearing in the denominators of the right-hand side of~\eqref{eq:evalerr} do not become too small over $[a,b]$, then we can expect the evaluation of $\widehat{r}$ to be accurate. Note that $|D(x)|$ is precisely the quantity examined in Section~\ref{sec:whybary}, where we argued that it takes values $O(1)$ or larger across the interval. %with appropriately chosen $t_k$. 
Further, since $r(x)\approx f(x)$ implies $|N(x)|\approx |D(x)f(x)|$, 
we see that $|N(x)|$ is not too small unless $|f(x)|$ is small. 
Put together, we expect the barycentric evaluation phase to be stable unless $|f(x)|$ (and hence $|r(x)|$) is small. 
Note that since~\eqref{eq:evalerr} measures the relative error, we usually cannot expect it to be $O(u)$ when $|r(x)|\approx |f(x)|\ll 1$.

\section{The rational Remez algorithm}\label{sec:remezbasics}

%\rrrr{\subsection{A review of the Remez algorithm}\label{subsec:remezreview}}
Initially developed by Werner~\cite{werner1962a,werner1962b} and Maehly~\cite{maehly1963methods}, the rational Remez algorithm extends the ideas of computing best polynomial approximations due to Remez~\cite{Rem34a,Rem34b}. It can be summarized as follows:
\smallskip

\begin{itemize}
	\item[\textbf{Step 1}] 	Set $k=1$ and choose $m+n+2$ distinct reference points
	\[
	a\leq x_0^{(k)}<\cdots<x_{m+n+1}^{(k)}\leq b.
	\]

\smallskip
	\item[\textbf{Step 2}] Determine the \emph{levelled error} $\lambda_k\in\mathbb{R}$ (positive or negative) and $r_k\in\mathcal{R}_{m,n}$ such that $r_k$ has no pole on $[a,b]$ and
	\begin{equation}\label{eq:equiosc}
	f(x_\ell^{(k)})-r_k(x_\ell^{(k)})=(-1)^{\ell+1}\lambda_k, \qquad \ell=0,\ldots,m+n+1.
	\end{equation}

\smallskip
	\item[\textbf{Step 3}] Choose as the next reference
$m+n+2$ local maxima $\{ x_\ell^{(k+1)}\}$ of $\left\vert f-r_k\right\vert$ such that
	\begin{equation}\label{eq:newaltern}
	s(-1)^\ell\left(f(x_\ell^{(k+1)})-r_k(x_\ell^{(k+1)})\right)\geq \left|\lambda_k\right|, \qquad \ell=0,\ldots,m+n+1,
	\end{equation}
	with $s\in\left\lbrace\pm 1\right\rbrace$ and such that for at least one $\ell\in\left\lbrace 0,\ldots,m+n+1\right\rbrace$, the left-hand side of~\eqref{eq:newaltern} equals $\left\Vert f-r_k\right\Vert_{\infty}$. If $r_k$ has converged to within a given threshold $\varepsilon_t>0$ (i.e., $(\left\Vert f-r_k\right\Vert_{\infty}-\lambda_k)/\left\Vert f-r_k\right\Vert_{\infty}\leq\varepsilon_t$~\cite[eq.~(10.8)]{powell1981approximation}) return $r_k$, else go to \textbf{Step 2} with $k\leftarrow k+1$.
\end{itemize}

If Step 2 is always successful, then convergence to the best approximation is assured~\cite[Theorem~9.14]{watson1980approximation}. 
%\rrrr{and the speed of convergence is generically quadratic~\cite{curtisosborne1966quadratic}}. 
It might happen that Step 2 fails, namely when all rational solutions satisfying the equations~\eqref{eq:equiosc} have poles in $[a,b]$. If the best approximation is non-degenerate and the initial reference set is already sufficiently close to optimal, then the algorithm will converge~\cite[\S V.6.B]{braess2012nonlinear}. To our knowledge, there is no effective way in general to determine when degeneracy is the cause of failure. 
%\rrrr{to find a pole-free solution of~\eqref{eq:equiosc}.}

We note that the rational Remez algorithm can also be adapted to work in the case of~\emph{weighted} best rational approximation. An early account of this is given in~\cite{curtis1966theory}. Given a positive weight function $w\in\mathcal{C}([a,b])$, the goal is to find $r^*\in\mathcal{R}_{m,n}$ such that the weighted error $\|f-r^*\|_{w,\infty}=\max_{x\in[a,b]}|w(x)(f(x)-r^*(x))|$ is minimal. Equations~\eqref{eq:equiosc} and~\eqref{eq:newaltern} get modified to
\[
	w(x_\ell^{(k)})\left(f(x_\ell^{(k)})-r_k(x_\ell^{(k)})\right)=(-1)^{\ell+1}\lambda_k, \qquad \ell=0,\ldots,m+n+1
\]
and
\[
	s(-1)^\ell w(x_\ell^{(k+1)})\left(f(x_\ell^{(k+1)})-r_k(x_\ell^{(k+1)})\right)\geq \left|\lambda_k\right|, \qquad \ell=0,\ldots,m+n+1,
\]
while the norm computations in Step 3 are taken with respect to $w$. Notice that the ability to work with the weighted error immediately allows us to compute the best approximation in the \emph{relative} sense, by taking $w(x)=1/|f(x)|$, assuming that $f$ is nonzero over $[a,b]$.

We discuss each step of the rational Remez algorithm in the following sections. We first address Step 2, as this is the core part where the barycentric representation is used. 
We then discuss initialization (Step 1) in Section~\ref{sec:initialization}, and finding the next reference set (Step 3) in Section~\ref{sec:nextrefs}. Our focus is on the unweighted setting, but we comment on how our ideas can be extended to the weighted case as well.

\section{Computing the trial approximation}\label{sec:main}
For notational simplicity, in this section we drop the index $k$ referring to the iteration number, the analysis being valid for any iteration of the rational Remez algorithm. We begin with the case $m=n$.

\subsection{Linear algebra in a polynomial basis}\label{sec:diag}
We first derive the Remez algorithm in an (arbitrary) polynomial basis.
At each iteration, we search for 
$r=p/q\in\mathcal{R}_{n,n}, p,q\in\mathbb{R}_n[x]$
 such that
\begin{equation}\label{eq:altern}
f(x_\ell)-r(x_\ell)=(-1)^{\ell+1}\lambda, \qquad \ell=0,\ldots,2n+1
\end{equation}
and assume that we represent $p$ and $q$ using a basis of polynomials $\varphi_0,\ldots,\varphi_n$ such that 
$\textnormal{span}_\mathbb{R}\left(\varphi_i\right)_{0\leq i\leq n}=\mathbb{R}_n[x]$:
\[p(x)=\sum_{k=0}^{n}c_{p,k}\varphi_k(x),\quad  q(x)=\sum_{k=0}^{n}c_{q,k}\varphi_k(x).\]
The linearized version of~\eqref{eq:altern} is then given by
\begin{equation*}
p(x_\ell)=q(x_\ell)\left(f(x_\ell)-(-1)^{\ell+1}\lambda\right),
\end{equation*}
which, in matrix form, becomes
\begin{equation}\label{eq:matrixaltern}
\Phi_xc_p = \left(\begin{bmatrix}
f(x_0) & & & \\
& f(x_1) & & \\
& & \ddots & \\
& & & f(x_{2n+1}) 
\end{bmatrix}-\lambda\begin{bmatrix}
-1 & & & \\
& 1 & & \\
& & -1 & \\
& & & \ddots
\end{bmatrix}\right)\Phi_xc_q,
\end{equation}
where $\Phi_x\in\mathbb{R}^{(2n+2)\times(n+1)}$ is the basis matrix
$\left(\Phi_x\right)_{\ell,k}=\varphi_k(x_\ell), 0\leq\ell\leq 2n+1,0\leq k\leq n$, and $c_p=[c_{p,0},c_{p,1},\ldots,c_{p,n}]^T$ and $c_q=[c_{q,0},c_{q,1},\ldots,c_{q,n}]^T$ are the coefficient vectors of $p$ and $q$. Note that in this paper, vector and matrix indices always start at zero. Up to multiplying both sides on the left by a nonsingular diagonal matrix $D=\diag\left(d_0,\ldots,d_{2n+1}\right)$, \eqref{eq:matrixaltern} can also be written as a generalized eigenvalue problem
\begin{equation}\label{eq:linaltern2}
\begin{bmatrix}
D\Phi_x & -FD\Phi_x
\end{bmatrix}\begin{bmatrix}
c_p \\
c_q
\end{bmatrix} = \lambda\begin{bmatrix}
0 & -SD\Phi_x
\end{bmatrix}\begin{bmatrix}
c_p \\
c_q
\end{bmatrix},
\end{equation}
with 
$F=\diag\left(f(x_0),\ldots,f(x_{2n+1})\right)$ and $S=\diag\left((-1)^{k+1}\right)$.

As described in Powell~\cite[Ch.~10.2]{powell1981approximation}, solving~\eqref{eq:linaltern2} is usually done by eliminating $c_p$. His presentation considers the monomial basis, but the approach is valid for any basis of $\mathbb{R}_n[x]$. By taking the full QR decomposition of $D\Phi_x$, we get
\[
D\Phi_x =\begin{bmatrix}
Q_1 & Q_2
\end{bmatrix}\begin{bmatrix}
R \\
0
\end{bmatrix} = Q_1R.
\]
Since $D\Phi_x$ is of full rank, we have $Q_1,Q_2\in\mathbb{R}^{(2n+2)\times (n+1)}$ and $Q_2^TQ_1=0$. By multiplying~\eqref{eq:linaltern2} on the left by $Q^T=\begin{bmatrix}
Q_1 & Q_2
\end{bmatrix}^T$, we 
%basically project onto $Q_1^\perp$, giving a
obtain a block triangular eigenvalue problem with 
lower-right $(n+1)\times (n+1)$ block
%which after simplifying on both sides gives
\begin{equation}\label{eq:geneigproj}
Q_2^TFQ_1Rc_q=\lambda Q_2^TSQ_1Rc_q.
\end{equation}
(The top-left $(n+1)\times (n+1)$ block has all eigenvalues at infinity, and is thus irrelevant.) 
In terms of polynomials, 
$(Q_1)_{\ell,k}=d_\ell\psi_k(x_\ell)$, $0\leq k\leq n, 0\leq\ell\leq 2n+1$,
 where $(\psi_k)_{0\leq k\leq n}$ is a family of orthonormal polynomials with respect to the discrete inner product
$\left\langle f,g\right\rangle_x=\sum_{k=0}^{2n+1}d_k^2f(x_k)\overline{g(x_k)}$. 
Moreover, if $(\varphi_k)_{0\leq k\leq n}$ is a degree-graded basis with $\deg\varphi_k=k$, then we have $\deg\psi_k=k,0\leq k\leq n$.

Let $\omega_x$ be the node polynomial associated with the reference nodes $x_0,\ldots,x_{2n+1}$, and 
$\Omega_x=\diag\left(1/\omega_x'(x_0),\ldots,1/\omega_x'(x_{2n+1})\right)$. We have~\cite[p.~114]{powell1981approximation}
\begin{equation}\label{eq:vandEq}
V_x^T\Omega_x V_x=0,
\end{equation}
where $V_x\in\mathbb{R}^{(2n+2)\times(n+1)}$ is the Vandermonde matrix associated with $x_0,\ldots,x_{2n+1}$, that is, $(V_x)_{i,j} = x_{i}^{j}$.
Indeed,
\[
\left(V_x^T\Omega_xV_x\right)_{i,j}=\sum_{\ell=0}^{2n+1}x_\ell^{i+j}\frac{1}{\omega_x'(x_\ell)}=(x^{i+j})[x_0,\ldots,x_{2n+1}]=0, \quad i,j\in\left\lbrace 0,\ldots,n\right\rbrace,
\]
the divided differences of order $2n+1$ of the function $x^{i+j}$ at the $\left\lbrace x_\ell\right\rbrace$ nodes, hence $0$ if $i+j\leq 2n$. 

By using the appropriate change of basis matrix in~\eqref{eq:vandEq}, we have
\begin{equation}\label{eq:genEq}
	\Phi_x^T\Omega_x \Phi_x=0.
\end{equation}
Now, by multiplying~\eqref{eq:linaltern2} on the left by $\Phi_x^T\Omega_xD^{-1}$ and using~\eqref{eq:genEq}, we can eliminate the $c_p$ term to obtain
\begin{equation}\label{eq:newprojaltern}
\Phi_x^T\Omega_x
%\rrrr{D^{-2}}
F\Phi_xc_q = \lambda \Phi_x^T\Omega_x
%\rrrr{D^{-2}}
S\Phi_xc_q. 
\end{equation}
Equation~\eqref{eq:newprojaltern} is the extension of~\cite[Eq.~(10.13)]{powell1981approximation} from the monomial basis to $\varphi_0,\ldots,\varphi_n$. Moreover, we have:
\begin{lemma}\label{lemma:1}
	The matrix $\Phi_x^T\Omega_xS\Phi_x$ is symmetric positive definite.
\end{lemma}
\begin{proof}
	Since $\Omega_xS=\left|\Omega_x\right|$, it means that $\Omega_xS$ is symmetric positive definite, and the conclusion follows. See also~\cite[Theorem~10.2]{powell1981approximation}.
\end{proof}

Since $\Phi_x^T\Omega_xF\Phi_x$ is also symmetric, it follows that all eigenvalues of~\eqref{eq:newprojaltern} are real and at most one eigenvector $c_q$ corresponds to a pole-free solution $r$ (i.e., $q$ has no root on $[a,b]$). To see this, suppose to the contrary that there exists another pole-free solution $r'$. Then, from~\eqref{eq:altern}, it follows that either $r(x_k)-r'(x_k)$ are all zero or they alternate in sign at least $2n+1$ times. In both cases, $r-r'\in\mathcal{R}_{2n,2n}$ has at least $2n+1$ zeros inside $[a,b]$, leading to $r=r'$.
%\begin{remark}
%	The equality
%	\[
%		Q_1^T\Omega_xD^{-2}Q_1=0,
%	\]
%	which follows immediately from~\eqref{eq:genEq}, also has an interpretation in terms of orthogonal functions. Indeed (!?), if we take the reproducing kernel of $\mathbb{R}_{2n+1}[x]$ with respect to the weighted inner product $\left\langle\cdot,\cdot\right\rangle_x$, we have
%	\[
%		d_k^2\omega_x'(x_k)\psi_{2n+1}(x_k)=c>0, \qquad k=0,\ldots,2n+1.
%	\]
%	Then
%	\[
%	(Q_1^T\Omega_xD^{-2}Q_1)_{\ell,j}=\sum_{k=0}^{2n+1}\frac{\psi_\ell(x_k)\psi_j(x_k)}{\omega_x'(x_k)}=\sum_{k=0}^{2n+1}d_k^2\frac{\psi_\ell(x_k)\psi_j(x_k)}{c}\psi_{2n+1}(x_k)=0,
%	\]
%	which follows from the orthogonality of $\psi_{2n+1}$ with all polynomials of degree $\leq 2n$ (with respect to $\left\langle\cdot,\cdot\right\rangle_x$). \bb{YN: don't really understand, please clarify}
%\end{remark}

We can in fact transform~\eqref{eq:geneigproj} into a symmetric eigenvalue problem (an observation which seems to date to~\cite{pelios1967rational}) by considering the choice $D=\left|\Omega_x\right|^{1/2}$, which leads to $Q_2=SQ_1$ in view of~\eqref{eq:genEq}. The system becomes $Q_1^TSFQ_1Rc_q = \lambda Q_1^TS^2Q_1Rc_q,$
which, by the change of variables $y=Rc_q$, gives
\begin{equation}\label{eq:simeig}
Q_1^TSFQ_1y=\lambda y.
\end{equation}

To get $c_p$, from~\eqref{eq:matrixaltern}, we have
$\left|\Omega_x\right|^{1/2}\Phi_xc_p=(F-\lambda S)\left|\Omega_x\right|^{1/2}\Phi_xc_q,$
or equivalently (by multiplication on the left by $Q_1^T$),
\begin{equation*}
Rc_p = Q_1^T(F-\lambda S)\left|\Omega_x\right|^{1/2}\Phi_xc_q = Q_1^TFQ_1y.
\end{equation*}
The vectors $Rc_p$ and $Rc_q$ can be seen as vectors of coefficients of the numerator and denominator of $r$ in the orthogonal basis $\psi_0,\ldots,\psi_n$. The (scaled) values of the denominator at each $x_k$ corresponding to an eigenvector $y$ can be recovered by computing 
\begin{equation}\label{eq:valqpoly}
\left|\Omega_x\right|^{1/2}\Phi_xc_q=Q_1y.  
\end{equation}
From this we can confirm the uniqueness of the pole-free solution:
since the eigenvectors are orthogonal, there is at most one generating a vector of denominator values of the same sign, making it the only pole-free solution candidate. 

\subsection{Linear algebra in a barycentric basis}\label{sec:bary}
An equivalent analysis is valid if we take $r$ in the barycentric form~\eqref{eq:baryr}. 
Namely,~\eqref{eq:altern} becomes
\begin{equation}\label{eq:baryaltern}
C\alpha = \left(\begin{bmatrix}
f(x_0) & & & \\
& f(x_1) & & \\
& & \ddots & \\
& & & f(x_{2n+1})
\end{bmatrix}-\lambda\begin{bmatrix}
-1 & & & \\
& 1 & & \\
& & -1 & \\
& & & \ddots
\end{bmatrix}\right)C\beta,
\end{equation}
where $C$ is now a $(2n+2)\times(n+1)$ Cauchy matrix with entries $C_{\ell,k}=1/(x_\ell-t_k)$ (we assume for the moment $\{x_\ell\}\cap\{t_k\}=\varnothing$)
and $\alpha=[\alpha_0,\alpha_1,\ldots,\alpha_n]^T$ and $\beta=[\beta_0,\beta_1,\ldots,\beta_n]^T$ are the column vectors of coefficients $\left\lbrace\alpha_k\right\rbrace$ and $\left\lbrace\beta_k\right\rbrace$. Again, this can be  transformed into a generalized eigenvalue problem
\begin{equation}\label{eq:geneigcauchy}
\begin{bmatrix}C & -FC\end{bmatrix}\begin{bmatrix}
\alpha \\
\beta
\end{bmatrix}=\lambda\begin{bmatrix}
0 & -SC
\end{bmatrix}\begin{bmatrix}
\alpha \\
\beta
\end{bmatrix}.
\end{equation}
To reduce~\eqref{eq:geneigcauchy} to a symmetric eigenvalue problem as in~\eqref{eq:simeig}, we form a link between the monomial and barycentric representations in terms of the basis matrices $V_x$ and $C$. We have:
\begin{lemma}\label{lemma:2}
Let 
$V_x$, $\omega_t$ be as defined above, and 
$V_t\in\mathbb{R}^{(n+1)\times(n+1)}$ be
  the Vandermonde matrix corresponding to the support points, i.e., $(V_t)_{i,j} = t_{i}^{j}$. Then
	\begin{equation*}\label{eq:barymonom}
	\diag\left(\frac{1}{\omega_t(x_0)},\ldots,\frac{1}{\omega_t(x_{2n+1})}\right)V_x=C\diag\left(\frac{1}{\omega_t'(t_0)},\ldots,\frac{1}{\omega_t'(t_{n})}\right)V_t.
	\end{equation*}
\end{lemma}
\begin{proof}
	If we look at an arbitrary element of the right-hand side matrix, we have
	\[
	\left(C\diag\left(\frac{1}{\omega_t'(t_0)},\ldots,\frac{1}{\omega_t'(t_n)}\right)V_t\right)_{j,\ell}=\sum_{k=0}^{n}\frac{1}{(x_j-t_k)\omega_t'(t_k)}t_k^\ell=\frac{x_j^\ell}{\omega_t(x_j)},
	\]
	where the second equality is a consequence of the Lagrange interpolation formula.
\end{proof}

%\begin{question}
%	We have
%	\[
%		C=\diag{\frac{1}{\omega_t(x_k)}}V_xV_t^{-1}\diag{\omega_t'(t_k)}.
%	\] Can the right hand side tell us anything on how to choose $t$ with respect to $x$?
%\end{question}

In place of $\Omega_x$ we will use the following matrix $\Delta$:
%$\Delta=\diag\left(\omega_t(x_k)^2\right)\Omega_x$:
\begin{lemma}\label{lemma:3}
	If $\Delta=\diag\left(\omega_t(x_0)^2,\ldots,\omega_t(x_{2n+1})^2\right)\Omega_x$, then $C^T\Delta C=0$.
\end{lemma}
\begin{proof} We apply Lemma~\ref{lemma:2} and use the fact that $V_x^T\Omega_xV_x=0$. Namely, 		
		$C^T\Delta C = \diag\left(\omega_t'(t_0),\ldots,\omega_t'(t_n)\right)V_t^{-T}V_x^T\Omega_xV_xV_t^{-1}\diag\left(\omega_t'(t_0),\ldots,\omega_t'(t_n)\right) = 0$.
\end{proof}

We now take the full QR decomposition of $\left\vert\Delta\right\vert^{1/2}C=(S\Delta)^{1/2}C$. We have
\begin{equation*}  %\label{eq:QRofDelC}
\left|\Delta\right|^{1/2}C=\begin{bmatrix} Q_1 & Q_2
\end{bmatrix}\begin{bmatrix}
R \\
0
\end{bmatrix}=Q_1R.
\end{equation*}
Based on Lemma~\ref{lemma:3}, we can again take $Q_2=SQ_1$. From~\eqref{eq:geneigcauchy} we get

\begin{equation*}
\begin{bmatrix}\left\vert\Delta\right\vert^{1/2}C & -F\left\vert\Delta\right\vert^{1/2}C\end{bmatrix}\begin{bmatrix}
\alpha \\
\beta
\end{bmatrix}=\lambda\begin{bmatrix}
0 & -S\left\vert\Delta\right\vert^{1/2}C
\end{bmatrix}\begin{bmatrix}
\alpha \\
\beta
\end{bmatrix}.
\end{equation*}
Multiplying this expression on the left by $\begin{bmatrix}
Q_1 & Q_2
\end{bmatrix}^T$ 
gives a block triangular matrix pencil, whose 
$(n+1)\times (n+1)$
lower-right corner is the barycentric analogue of~\eqref{eq:geneigproj}: $Q_2^TFQ_1R\beta = \lambda Q_2^TSQ_1R\beta.$
After substituting $Q_2^T=Q_1^TS$, we get
\begin{equation}  \label{eq:symbary}
Q_1^T(SF)Q_1R\beta=\lambda Q_1^TS^2Q_1R\beta,  
\end{equation}
which, by the change of variable $y=R\beta$, becomes a standard symmetric eigenvalue problem in $\lambda$ with eigenvector $y$ (recall that $S,F$ are diagonal):
\begin{equation}\label{eq:retbeta}
Q_1^T(SF)Q_1y=\lambda y.
%\underbrace{Q_1^T(SF)Q_1}_{=A}y=\lambda y.
\end{equation} 
%The matrix $A$ is symmetric, 
Hence, computing its eigenvalues is a well-conditioned operation. The values of the denominator of the rational interpolant corresponding to each eigenvector $y$ can be recovered by computing 
\begin{equation}  \label{eq:valqbary}
\diag\left(\omega_t(x_0),\ldots,\omega_t(x_{2n+1})\right)C\beta=\diag\left(\omega_t(x_0),\ldots,\omega_t(x_{2n+1})\right)\left|\Delta\right|^{-1/2}Q_1y.  
\end{equation}
%$\diag{\omega_t(x_k)}C\beta=\diag{\omega_t(x_k)}Q_1y$. 

As in the polynomial case, 
there is at most one solution such that 
$q(x)=D(x)\omega_t(x)$ has no root in $[a,b]$; indeed,~\eqref{eq:valqpoly} and~\eqref{eq:valqbary} represent the values of $q(x_\ell)$ for $r=p/q$ and $x_\ell$ satisfying equation \eqref{eq:altern}. We use this sign test involving~\eqref{eq:valqbary} to determine the levelled error $\lambda$ that gives a pole-free $r$ in Step 2 of our rational Remez algorithm. The appropriate $\beta$ is then taken by solving $R\beta=y$.
%(in exact arithmetic) the condition  since we are solving the exact same problem, there can be only one same-sign solution.. making it the only candidate for a pole-free solution. }
%(YN: is this obvious? )
%Since the eigenvectors are orthogonal, there should be at most one vector on the left-hand side with entries of the same sign, making it the only candidate for a pole-free solution. (YN: is this obvious? )
From~\eqref{eq:baryaltern}, we have
\[
\left\vert\Delta\right\vert^{1/2}C\alpha=(F-\lambda S)\left\vert\Delta\right\vert^{1/2}C\beta,
\]
or equivalently (by multiplication on the left by $Q_1^T$)
\begin{equation}\label{eq:retalpha}
R\alpha=Q_1^T(F-\lambda S)\left\vert\Delta\right\vert^{1/2}C\beta=Q_1^T(F-\lambda S)Q_1y=Q_1^TFQ_1y,
\end{equation}
which allows us to recover $\alpha$ (and thus $r$).

Most of the derivations in this section can be carried over to the weighted approximation setting as well. In particular, the reader can check that the weighted versions of Equations~\eqref{eq:geneigcauchy} and~\eqref{eq:retbeta} correspond to
\[
\begin{bmatrix}C & -FC\end{bmatrix}\begin{bmatrix}
\alpha \\
\beta
\end{bmatrix}=\lambda\begin{bmatrix}
0 & -SW^{-1}C
\end{bmatrix}\begin{bmatrix}
\alpha \\
\beta
\end{bmatrix}
\]
and
\[
Q_1^T(SF)Q_1y=\lambda Q_1^TW^{-1}Q_1y,
\]
where $W=\diag\left(w(x_0),\ldots,w(x_{2n+1})\right)$ and all the other quantities are the same as before. While not leading to a symmetric eigenvalue problem, the symmetric and symmetric positive definite matrices appearing in the second pencil seem to suggest that the eigenproblem computations will again correspond to well-conditioned operations. Our experiments support this statement and we leave it as future work to make this rigorous. To recover $\alpha$,~\eqref{eq:retalpha} becomes $R\alpha = Q_1^T(F-\lambda SW^{-1})Q_1y$.

\subsection{Conditioning of the QR factorization}\label{sec:qrcond}
Since the above discussion makes heavy use of the matrix $Q_1$, it is desirable that computing the (thin) QR factorization $\left\vert\Delta\right\vert^{1/2}C=Q_1R$ is a well-conditioned operation. 

Here we examine the conditioning of $Q_1$, the orthogonal factor in the QR factorization of $|\Delta|^{1/2}C$, as this is the key matrix for constructing~\eqref{eq:symbary}. 
We use the fact that the standard Householder QR algorithm is invariant under column scaling, that is, it computes the same $Q_1$ for both $\left\vert\Delta\right\vert^{1/2}C$ and $\left\vert\Delta\right\vert^{1/2}C\Gamma$ for diagonal $\Gamma$~\cite[Ch.~19]{Higham:2002:ASNA}. We thus consider 
\begin{equation}  \label{eq:mincd}
\min_{\Gamma\in\mathcal{D}_{n+1}}\kappa_2(\left\vert\Delta\right\vert^{1/2}C\Gamma),   
\end{equation}
where $\mathcal{D}_{n+1}$ is the set of $(n+1)\times (n+1)$ diagonal matrices. 
We have 

\begin{theorem}\label{thm:bernd}
Let $t_k\in(x_{2k},x_{2k+1})$ for $k=0,\ldots,n$ and $s_k\in(x_{2k+1},x_{2k+2})$ for $k=0,\ldots,n-1$, $s_n\in(x_{2n+1},\infty)$,  and define $\omega_s(x)=\prod_{k=0}^n(x-s_k)$. Then 
  \begin{equation}\label{eq:berndthm}
\min_{\Gamma\in\mathcal{D}_{n+1}}\kappa_2(\left\vert\Delta\right\vert^{1/2}C\Gamma)    \leq \max_\ell\sqrt{\left|\frac{\omega_s(x_\ell)}{\omega_t(x_\ell)}\right|}\cdot \max_k\sqrt{\left|\frac{\omega_t(x_k)}{\omega_s(x_k)}\right|}. 
  \end{equation}
\end{theorem}

\begin{proof}
Let $\{y_j\}$ be a $(2n+2)$-element set such that $y_j\in(x_j,x_{j+1}), j=0,\ldots,2n$, $y_{2n+1} > x_{2n+1}$ and let $C_{x,y}\in\mathbb{R}^{(2n+2)\times (2n+2)}$ be the Cauchy matrix with elements $(C_{x,y})_{j,k}=1/(x_j-y_k)$.
If we consider $D_1=\mbox{diag}(\sqrt{\left|\omega_y(x_j)/\omega_x'(x_j)\right|})$ and $D_2=\mbox{diag}(\sqrt{\left|\omega_x(y_j)/\omega_y'(y_j)\right|})$,
then the matrix $D_1C_{x,y}D_2$ is orthogonal. This follows, for instance, if we examine the elements of its associated Gram matrix $G$ and use divided differences. Indeed, for an arbitrary element $(G)_{j,k}$ with $j\neq k$, we have
\begin{align*}
-(G)_{j,k}&= \sqrt{\left|\frac{\omega_x(y_j)\omega_x(y_k)}{\omega_y'(y_j)\omega_y'(y_k)}\right|}\sum_{\ell=0}^{2n+1}\frac{\omega_y(x_\ell)}{(x_\ell-y_j)(x_\ell-y_k)\omega_x'(x_\ell)} \\
&=\sqrt{\left|\frac{\omega_x(y_j)\omega_x(y_k)}{\omega_y'(y_j)\omega_y'(y_k)}\right|}\left(\frac{\omega_y(x)}{(x-y_j)(x-y_k)}\right)[x_0,\ldots,x_{2n+1}]=0.
\end{align*}
Similarly, since $\prod_{j\neq k}(x-y_j)=q(x)(x-y_k)+\omega_y'(y_k)$, with $q\in\mathbb{R}_{2n}[x]$, we have,
\begin{align*}
-(G)_{k,k}&=\frac{\omega_x(y_k)}{\omega_y'(y_k)}\sum_{\ell=0}^{2n+1}\frac{\omega_y(x_\ell)}{(x_\ell-y_k)^2\omega_x'(x_\ell)}
=\frac{\omega_x(y_k)}{\omega_y'(y_k)}\left(\frac{\prod_{j\neq k}(x-y_j)}{x-y_k}\right)[x_0,\ldots,x_{2n+1}]\\
&= \frac{\omega_x(y_k)}{\omega_y'(y_k)}\left(q(x)+\frac{\omega_y'(y_k)}{x-y_k}\right)[x_0,\ldots,x_{2n+1}] \\
&= \omega_x(y_k)\left(\frac{1}{x-y_k}\right)[x_0,\ldots,x_{2n+1}] = \omega_x(y_k)\frac{-1}{\omega_x(y_k)}=-1.
\end{align*}

Now, if we take $t_k=y_{2k}, s_k=y_{2k+1},$ for $k=0,\ldots,n$, there exist $D\in\mathcal{D}_{2n+2}$ and $\Gamma\in\mathcal{D}_{n+1}$ such that $\left\vert\Delta\right\vert^{1/2}C\Gamma=DD_1C_{x,y}D_2I_t$, where $D=\mbox{diag}(\sqrt{\left\vert\omega_t(x_j)/\omega_s(x_j)\right\vert})$ and $I_t$ is obtained by removing every second column from $I_{2n+2}$. In particular, $\Gamma=I_t^TD_2I_t$. It follows that
$$\kappa_2(\left\vert\Delta\right\vert^{1/2}C\Gamma)\leq\kappa_2(D)=\max_\ell\sqrt{\left|\frac{\omega_s(x_\ell)}{\omega_t(x_\ell)}\right|}\cdot \max_k\sqrt{\left|\frac{\omega_t(x_k)}{\omega_s(x_k)}\right|}.$$
\end{proof}

Let $\Gamma=I_t^TD_2I_t$ be as in the proof of Theorem~\ref{thm:bernd}. It turns out that for the choice $t_k=x_{2k+1}-\varepsilon, s_k=x_{2k+1}+\varepsilon$, for $k=0,\ldots,n$, as $\varepsilon\rightarrow 0$, the matrix $\left\vert\Delta\right\vert^{1/2}C$ has a finite limit $\widetilde{C}$ of full column rank, and similarly $\Gamma$ tends to some diagonal matrix $\widetilde{\Gamma}$ with positive diagonal entries. From Theorem~\ref{thm:bernd} and its proof we know that $\widetilde{C}\widetilde{\Gamma}$ has condition number 1, and, more precisely, orthonormal columns. We thus obtain an explicit thin QR decomposition of $\widetilde{C}$ (by direct calculation):

	\begin{corollary}\label{cor:bernd} In the limit $t_k\nearrow x_{2k+1}$, for $k=0,\ldots,n$, the matrix $\left\vert\Delta\right\vert^{1/2}C$ converges to $\widetilde{C}$, with entries 
$$
	(\widetilde C )_{j,k} = \left\{\begin{array}{ll}
\frac{|w'_{t}(t_k)|}{\sqrt{|w_x'(t_k)|}}
 & \textnormal{ if } j=2k+1, \\
	0 & \textnormal{ if } j=2\ell+1, \ell\neq k, \\
\frac{|w_{t}(x_j)|}{\sqrt{|w_x'(x_j)|}}
/(x_j-t_k)
 & \textnormal{ if } j=2\ell,
	\end{array}
	\right.
$$
% verify via MATLAB  codes/remezfast_bernd.m
and explicit thin QR decomposition $\widetilde{C}=Q_1R$, where
	\begin{equation*}
	(Q_1)_{j,k} = \left\{\begin{array}{ll}
	1/\sqrt{2} & \textnormal{ if } j=2k+1, \\
	0 & \textnormal{ if } j=2\ell+1, \ell\neq k, \\
\left|\frac{w_{t}(x_j)}{w'_{t}(t_k)}\right|\sqrt{\left|\frac{w_x'(t_k)}{2w_x'(x_j)}\right|}/(x_j-t_k)
%	\sqrt{\left\vert\dfrac{\omega_t(x_{2\ell})\widetilde{\omega}(t_k)}{2\widetilde{\omega}'(x_{2\ell})\omega_t'(t_k)}\right\vert}\cdot\dfrac{1}{x_{2\ell}-t_k} 
& \textnormal{ if } j=2\ell,
	\end{array}
	\right.
	\end{equation*}
and 
$	R = \sqrt{2}\ \textnormal{diag}\left(
\frac{|w'_{t}(t_0)|}{\sqrt{|w_x'(t_0)|}}, \ldots, 
\frac{|w'_{t}(t_n)|}{\sqrt{|w_x'(t_n)|}}\right).$ 

%	\begin{equation*}
%	R = \textnormal{diag}\left(\sqrt{2}\sqrt{\left\vert\frac{\omega_t'(t_0)}{\widetilde{\omega}(t_0)}\right\vert},\ldots,\sqrt{2}\sqrt{\left\vert\frac{\omega_t'(t_n)}{\widetilde{\omega}(t_n)}\right\vert}\right), \quad \widetilde{\omega}(x)=\frac{\omega_x(x)}{\omega_t(x)}=\prod_{j=0}^{n}(x-x_{2j}),
%	\end{equation*}
%	and
%	$$
%	(Q_1)_{j,k} = \left\{\begin{array}{ll}
%	1/\sqrt{2} & \textnormal{ if } j=2k+1, \\
%	0 & \textnormal{ if } j=2\ell+1, \ell\neq k, \\
%	\sqrt{\left\vert\dfrac{\omega_t(x_{2\ell})\widetilde{\omega}(t_k)}{2\widetilde{\omega}'(x_{2\ell})\omega_t'(t_k)}\right\vert}\cdot\dfrac{1}{x_{2\ell}-t_k} & \textnormal{ if } j=2\ell.
%	\end{array}
%	\right.
%	$$
	\end{corollary}

Corollary~\ref{cor:bernd} suggests the choice
	\begin{equation}\label{eq:supportchoose}
		t_k=x_{2k+1} \textnormal{ for }k=0,\ldots,n.
	\end{equation}
This takes us back to the interpolatory mode of barycentric representations~\eqref{eq:rzbary},  in which we take $\alpha_k=\beta_k(f(t_k)-\lambda)$ for all $k$, instead of solving the system~\eqref{eq:retalpha}. This interpolatory mode formulation is used in~\cite[Sec.~3.2.3]{ionita2013lagrange}. Our derivation provides a theoretical justification by showing that it is optimal with respect to the conditioning of $\left\vert\Delta\right\vert^{1/2}C\Gamma$. 
Moreover, since $\min_{\Gamma\in\mathcal{D}_{n+1}}\kappa_2(\widetilde C\Gamma)=1$ in~\eqref{eq:mincd}, 
forming the QR factorization of $\left\vert\Delta\right\vert^{1/2}C$ via a standard algorithm (e.g.~Householder QR) to obtain $Q_1$ is actually unnecessary, as the explicit form of $Q_1$ is given in Corollary~\ref{cor:bernd}. 
In addition, we reduce the problem to a symmetric eigenvalue problem~\eqref{eq:retbeta}, resulting in well-conditioned eigenvalues, with $\beta$ being obtained by solving the diagonal system $R\beta=y$ with $y$ as in~\eqref{eq:retbeta}. Compared to~\eqref{eq:altern}, where we want $q$ to have the same sign over $\{x_\ell\}$, we similarly require that $\beta$ and thus $y$ have components alternating in sign, which uniquely fixes the norm 1 eigenvector $y$ in~\eqref{eq:retbeta}. Our approach also allows for nondiagonal types, as we describe next. 
% \rr{describe in Section~\ref{sec:nondiag}}.

\subsection{The nondiagonal case $\boldmath{m\neq n}$}\label{sec:nondiag}
%In the nondiagonal case where
As pointed out in Section~\ref{subsec:nondiag}, when searching for a best approximant with $m\neq n$, we need to force the coefficient vector $\alpha$ or $\beta$ to lie in a certain subspace. 
This results in modified versions of~\eqref{eq:geneigcauchy}. Namely,
\begin{equation}\label{eq:geneigcauchynondiagm>n}
\begin{bmatrix}C & -FCP_n\end{bmatrix}\begin{bmatrix}
\alpha \\
\widehat\beta
\end{bmatrix}=\lambda\begin{bmatrix}
0 & -SCP_{n}
\end{bmatrix}\begin{bmatrix}
\alpha \\
\widehat\beta
\end{bmatrix},\qquad \mbox{when\ } m>n, 
\end{equation}
for $\widehat\beta\in\mathbb{C}^{n+1}$,
and we take $\beta=P_n\widehat\beta$. 
Similarly, 
\begin{equation}\label{eq:geneigcauchynondiagm<n}
\begin{bmatrix}CP_m & -FC\end{bmatrix}\begin{bmatrix}
\widehat\alpha \\
\beta
\end{bmatrix}=\lambda\begin{bmatrix}
0 & -SC
\end{bmatrix}\begin{bmatrix}
\widehat\alpha \\
\beta
\end{bmatrix},\qquad \mbox{when\ } m<n,
\end{equation}
for $\widehat\alpha\in\mathbb{C}^{m+1}$,
and we take $\alpha=P_m\widehat\alpha$. 

Below we describe the reduction of the generalized eigenvalue problems~\eqref{eq:geneigcauchynondiagm>n} and \eqref{eq:geneigcauchynondiagm<n} to standard symmetric eigenvalue problems. 

% this corresponds to setting the leading $m-n$ coefficients of $D(x)\prod_{k=0}^n(x-t_k)$ to be 0, 
%$P_n$ is of size $(m+n+2)\times (m+n+2-(m-n))$, and $P_m$ is $(m+n+2)\times (m+n+2-(n-m))$.

\paragraph{Case $m>n$} In this case, $C\in\mathbb{R}^{(m+n+2)\times(m+1)}$. 
Since $\det|\Delta|^{1/2}\neq 0$,
% (the other case is analogous as we shall see). 
\eqref{eq:geneigcauchynondiagm>n} is equivalent to the generalized eigenvalue problem
\begin{equation}
\label{eq:GEPm>n}
\begin{bmatrix}\left\vert\Delta\right\vert^{1/2}C & -F\left\vert\Delta\right\vert^{1/2}CP_n\end{bmatrix}\begin{bmatrix}
\alpha \\
\widehat\beta
\end{bmatrix}=\lambda\begin{bmatrix}
0 & -S\left\vert\Delta\right\vert^{1/2}CP_n
\end{bmatrix}\begin{bmatrix}
\alpha \\
\widehat\beta
\end{bmatrix}.  
\end{equation}
Consider the (thin) QR decomposition of 
$\left\vert\Delta\right\vert^{1/2}C
\begin{bmatrix}
P_n& P_n^\perp  
\end{bmatrix}=(S\Delta)^{1/2}C
\begin{bmatrix}
P_n&P_n^\perp  
\end{bmatrix}$:
\[
\left|\Delta\right|^{1/2}C
\begin{bmatrix}
P_n& P_n^\perp  
\end{bmatrix}
=\begin{bmatrix} Q_1 & Q_2
\end{bmatrix}R=
\begin{bmatrix} Q_1 & Q_2
\end{bmatrix}
\begin{bmatrix}
R_1& R_{12}  \\ 0& R_2
\end{bmatrix}.
\]
Then we have the identity $
\begin{bmatrix}
Q_1& Q_2  
\end{bmatrix}^T(SQ_1)=0$, 
as can be verified analogously to~\eqref{eq:vandEq} using divided differences. 
%Noting that $\begin{bmatrix}Q_1& Q_2  \end{bmatrix}^T(SQ_1)=0$
This implies $(SQ_1)^T\left\vert\Delta\right\vert^{1/2}C=0$, 
so by left-multiplying \eqref{eq:GEPm>n} by $
\begin{bmatrix}
(SQ_1)^\perp& SQ_1  
\end{bmatrix}^T$ we obtain a block upper-triangular eigenvalue problem with lower-right $(n+1)\times(n+1)$ block 
\begin{equation*}
(SQ_1)^TFQ_1R_{1}\widehat\beta = \lambda (SQ_1)^TSQ_1R_1\widehat\beta, 
\end{equation*}
which again reduces to the standard symmetric eigenvalue problem 
(setting $y=R_1\widehat\beta$)
\begin{equation}\label{eq:retbetamgn}
Q_1^T(SF)Q_1y=\lambda y.
\end{equation} 

From \eqref{eq:GEPm>n}, we have
$
\left\vert\Delta\right\vert^{1/2}C\alpha=(F-\lambda S)\left\vert\Delta\right\vert^{1/2}CP_n\widehat\beta$. 
Left-multiplying by
$
\begin{bmatrix}
Q_1&Q_2  
\end{bmatrix}^T$ and using 
$
\begin{bmatrix}
Q_1&Q_2  
\end{bmatrix}^TS\left\vert\Delta\right\vert^{1/2}CP_n=0$, we obtain
\begin{equation*}
\begin{split}
R
\begin{bmatrix}
P_n&P_n^\perp  
\end{bmatrix}^T\alpha=&
\begin{bmatrix}
Q_1&Q_2  
\end{bmatrix}^TF\left\vert\Delta\right\vert^{1/2}CP_n\widehat\beta
=
\begin{bmatrix}
Q_1&Q_2  
\end{bmatrix}^TFQ_1R_1\widehat\beta \\ =&
\begin{bmatrix}
Q_1&Q_2  
\end{bmatrix}^TFQ_1y.
\end{split}
\end{equation*}
Therefore 
\begin{equation*}
%\label{eq:alpham>n}
\alpha=
\begin{bmatrix}
P_n&P_n^\perp  
\end{bmatrix}
R^{-1}
\begin{bmatrix}
Q_1&Q_2  
\end{bmatrix}^TFQ_1y,
\end{equation*}
which is obtained by computing the vector
$\widehat y=
\begin{bmatrix}
Q_1&Q_2  
\end{bmatrix}^TFQ_1y$, then solving $R\widetilde y = \widehat y$ for $\widetilde y$, 
then $\alpha =  
\begin{bmatrix}
P_n&P_n^\perp  
\end{bmatrix}
\widetilde y$. 
%where $R^{-1}$ is of course solved by back substitution. 

\paragraph{Case $m<n$}
This case is analogous to the previous one; we highlight the differences.  
$C$ is a $(m+n+2)\times (n+1)$ matrix. 
Equation~\eqref{eq:geneigcauchynondiagm<n} is equivalent to 
%the generalized eigenvalue problem
\begin{equation}
\label{eq:GEPm<n}
\begin{bmatrix}\left\vert\Delta\right\vert^{1/2}CP_m & -F\left\vert\Delta\right\vert^{1/2}C\end{bmatrix}\begin{bmatrix}
\widehat\alpha \\
\beta
\end{bmatrix}=\lambda\begin{bmatrix}
0 & -S\left\vert\Delta\right\vert^{1/2}C
\end{bmatrix}\begin{bmatrix}
\widehat\alpha \\
\beta
\end{bmatrix}.  
\end{equation}
Consider the (thin) QR decompositions 
%of $\left\vert\Delta\right\vert^{1/2}C=(S\Delta)^{1/2}C$ and 
%Also consider 
\[
\left|\Delta\right|^{1/2}C=(S\Delta)^{1/2}C
%=\begin{bmatrix} Q_1 & Q_2\end{bmatrix}\begin{bmatrix}R \\0\end{bmatrix}=
=Q_1R,\qquad 
\left|\Delta\right|^{1/2}CP_m
%=\begin{bmatrix}\widehat Q_1 & \widehat Q_2\end{bmatrix}\begin{bmatrix}\widehat R \\0\end{bmatrix}
=\widehat Q_1\widehat R.
\] 
Here $Q_1\in\mathbb{R}^{(m+n+2)\times(n+1)},\widehat Q_1\in\mathbb{R}^{(m+n+2)\times(m+1)}$.
We have $\widehat Q_1^T(SQ_1)=0$, which again can be established using divided differences.
% holds (of course, the $Q_i$ here are different from those in the previous subsection). 
%Since $\widehat Q_1^T(SQ_1)=0$ 
This implies $(SQ_1)^T\left\vert\Delta\right\vert^{1/2}CP_m=0$, so 
left-multiplying equation \eqref{eq:GEPm<n} by $
\begin{bmatrix}
(SQ_1)^\perp& SQ_1  
\end{bmatrix}^T$ 
results in 
a block upper-triangular eigenvalue problem with lower-right block 
\[
(SQ_1)^TFQ_1R\beta = \lambda (SQ_1)^TSQ_1R\beta, 
\]
which also reduces to the standard symmetric eigenvalue problem 
(setting $y=R\beta$)
\begin{equation}\label{eq:retbetammn}
Q_1^T(SF)Q_1y=\lambda y.
\end{equation} 

From~\eqref{eq:GEPm<n}, we have
$\left\vert\Delta\right\vert^{1/2}CP_m\widehat\alpha=(F-\lambda S)\left\vert\Delta\right\vert^{1/2}C\beta$. 
Left-multiplying by $\widehat Q_1^T$ and using 
$
\widehat Q_1^TS\left\vert\Delta\right\vert^{1/2}C=0$, we obtain
\begin{equation*}
\begin{split}
\widehat R
\widehat \alpha= &
\widehat Q_1^T
F\left\vert\Delta\right\vert^{1/2}C\beta
= 
\widehat Q_1^T
F Q_1 R\beta
=\widehat Q_1^T
F Q_1 y.
% \\ =&\begin{bmatrix}Q_1&Q_2  \end{bmatrix}^TFQ_1y,
\end{split}
\end{equation*}
Therefore 
\begin{equation*}
%\label{eq:alpham<n}
\widehat\alpha=
\widehat R^{-1}
\widehat Q_1^TFQ_1y,
\end{equation*}
obtained via 
$\widehat y=
\widehat Q_1^TFQ_1y$, then solving the linear system $\widehat R\widehat\alpha = \widehat y$. 

Analogously to our comments at the end of Section~\ref{sec:bary}, the analysis for nondiagonal approximation presented here carries over to the weighted setting. In both the $m>n$ and $m<n$ scenarios, the standard symmetric eigenproblems~\eqref{eq:retbetamgn} and~\eqref{eq:retbetammn} become
\[
	Q_1^T(SF)Q_1y=\lambda Q_1^TW^{-1}Q_1y,
\]
where $y=R_1\widehat\beta$ when $m>n$ and $y=R\beta$ when $m<n$. Recovering the set of barycentric coefficients in the numerator corresponds to solving the systems
\[
	\alpha = \begin{bmatrix}
	P_n&P_n^\perp  
	\end{bmatrix}
	R^{-1}
	\begin{bmatrix}
	Q_1&Q_2  
	\end{bmatrix}^T(F-\lambda SW^{-1})Q_1y, \qquad m > n
\]
and
\[
\widehat\alpha=
\widehat R^{-1}
\widehat Q_1^T(F-\lambda SW^{-1})Q_1y, \qquad m < n.
\]

\paragraph{Stability and conditioning}
%In Section~\ref{sec:main} 
We have just shown that the matrices arising in our rational Remez algorithm have explicit expressions, and the eigenvalue problem reduces to a standard symmetric problem. 
Indeed, our experiments corroborate that we have greatly improved the stability and conditioning of the rational Remez algorithm using the barycentric representation. 
However, the algorithm is still not guaranteed to compute $r^*$ to machine precision. 
%of the linear algebra problems involved in our algorithm 
Let us summarize the situation for the unweighted case. 
As shown in Corollary~\ref{cor:bernd}, the computation of $Q_1$ can be done explicitly, and the linear system $y=R\beta$ is diagonal, hence can be solved with high relative accuracy. The main source of numerical errors is therefore in the symmetric eigenvalue problem~\eqref{eq:retbeta}, \eqref{eq:retbetamgn} or~\eqref{eq:retbetammn}. 
As is well known, by Weyl's bound~\cite[Cor.~IV.4.9]{stewart-sun:1990}, eigenvalues of symmetric matrices are well conditioned with condition number 1; thus $\lambda$ is computed with $O(u)$ accuracy, assuming for simplicity that $\|f\|_\infty=1$ (without loss of generality). The eigenvector, on the other hand, has conditioning $O(1/\mbox{gap})$~\cite[Ch.~V]{stewart-sun:1990}, where $\mbox{gap}$ is the distance between the desired $\lambda$ and the rest of the eigenvalues. These eigenvalues are equal to those of the nonzero eigenvalues of the generalized eigenproblem~\eqref{eq:linaltern2}, and are inherent in the Remez algorithm, i.e., they cannot be changed e.g. by  a change of bases. For a fixed $f$, $\mbox{gap}$ tends to decrease as $m,n$ increase, and we typically have $\mbox{gap}=O(|\lambda|)$. Hence the computed eigenvector tends to have accuracy $O(u/|\lambda|)$, and if the eigenvector $y$ has small elements, the componentwise relative accuracy may be worse. 
The computation therefore breaks down (perhaps as expected) when $|\lambda|=O(u)$, that is, when the error curve has amplitude of size machine precision. 
%which loses all digits once $|\lambda|=O(u)$, that is, when the error curve has amplitude of size machine precision. 
%loses relative precision precisely when $|\lambda|=O(u)$. 
%and cannot be changed by equivalence transformations. 

\subsection{Adaptive choice of the support points}\label{sec:supportpts}
Theorem~\ref{thm:bernd} gives an optimal choice of support points $t_k=x_{2k+1}$ in terms of optimizing $\min_{\Gamma\in\mathcal{D}_{n+1}}\kappa_2(\left\vert\Delta\right\vert^{1/2}C\Gamma)$.  In Section~\ref{sec:whybary} we discussed another desideratum for the support points $\{t_k\}$: the resulting $|D(x_\ell)|=|q(x_\ell)\prod_{k=0}^n(x_\ell-t_k)|$ should take uniformly large values for all $\ell$. Fortunately, this requirement is also met with this choice, as was illustrated in Figure~\ref{fig:plotq}.

When $m\neq n$,~\eqref{eq:supportchoose} does not determine enough support points. We take the remaining $|m-n|$ support points from the rest of the reference points in Leja style, i.e.,~to maximize the product of the differences (see for instance~\cite[p.~334]{reichel1990newton}). This is a heuristic strategy, and the optimal choice is a subject of future work: indeed, in this case $\min_{\Gamma\in\mathcal{D}_{n+1}}\kappa_2(\left\vert\Delta\right\vert^{1/2}CP_{m,n}\Gamma)>1$.

%\rrrr{Let us comment more on the choice of support points $\{t_k\}$. 
%In the barycentric representation of rational functions as in~\eqref{eq:baryr}, it is of crucial importance (for numerical stability) how to pick 
%$\{t_k\}$. 
%For example, when $f$ has a (near) singularity, typically we would like to cluster the support points near that singularity. 
%In AAA, they are chosen by a greedy strategy (see~\cite[Section~3]{aaapreprint} for the specifics), which tends to find an appropriate distribution. For our barycentric Remez algorithm, one choice would be to use the same support points as the output of AAA. However, for high-degree approximants, we have observed that the above adaptive choice~\eqref{eq:supportchoose} gives superior stability. 
%}
%\bb{YN: this para might be out of place.}

\section{Initialization}\label{sec:initialization}
%\rrrr{Since the computation of the rational approximation at an arbitrary iteration was discussed at length above (Step 2 from \rr{the beginning of the} \rrr{Section~\ref{sec:remezbasics}}{section}), the focus here is on finding an initial set of reference points (Step 1). We treat how these points get updated at each iteration (Step 3) in Section~\ref{sec:nextrefs}.}

An indispensable component of a successful Remez algorithm implementation is a method for finding a good set of initial reference points $\{x_\ell \}$. 
%\rrrr{In this subsection we describe our approach to finding such reference points.} 
A key element of our approach is the AAA-Lawson algorithm, which can efficiently find an approximate solution to the minimax problem~\eqref{eq:mainproblem} (to low accuracy).

\subsection{Carath\'{e}odory-Fej\'{e}r (CF) approximation}\label{sec:inicf}
We first attempt to compute the CF approximant~\cite{trefethen1983caratheodory,van2011robust} to $f$, and use it to find the initial reference points (as explained in Section~\ref{sec:nextrefs}). 
The dominant computation is an SVD of a Hankel matrix of Chebyshev coefficients, which usually does not cause a computational bottleneck.
This method was also used in the previous Chebfun~\texttt{remez} code. When $f$ is smooth, the result produced by CF approximation is often indistinguishable from the best approximation, but nonsmooth cases may be very different.

\subsection{AAA-Lawson approximation}\label{subsubsec:AAAlaw}
This approach is based on the AAA algorithm~\cite{aaapreprint} followed by 
an adaptation of the Lawson algorithm. 
%This is in some sense a new algorithm for the rational minimax problem that can efficiently find an approximate (to low accuracy). 
The resulting algorithm is also based crucially on the barycentric representation.
To keep the focus on Remez, we defer the details to Section~\ref{sec:iniAAA}.

The output of the AAA-Lawson iteration typically has a nearly equioscillatory error curve $e=f-r$, from which we find the initial set of reference points as the extrema of $e$. For the prototypical example $f=|x|$, AAA-Lawson initialization lets our barycentric~\texttt{minimax} code converge for type up to $(40,40)$. 
The entire process relies on a moderate number of SVDs (say $\max(m,n)+10$). 
% which is seldom a bottleneck in the overall algorithm. 

\begin{figure}[h!]
	\centering

	\begin{minipage}[t]{0.44\hsize}
		\centering
		\includegraphics[width=0.95\textwidth]{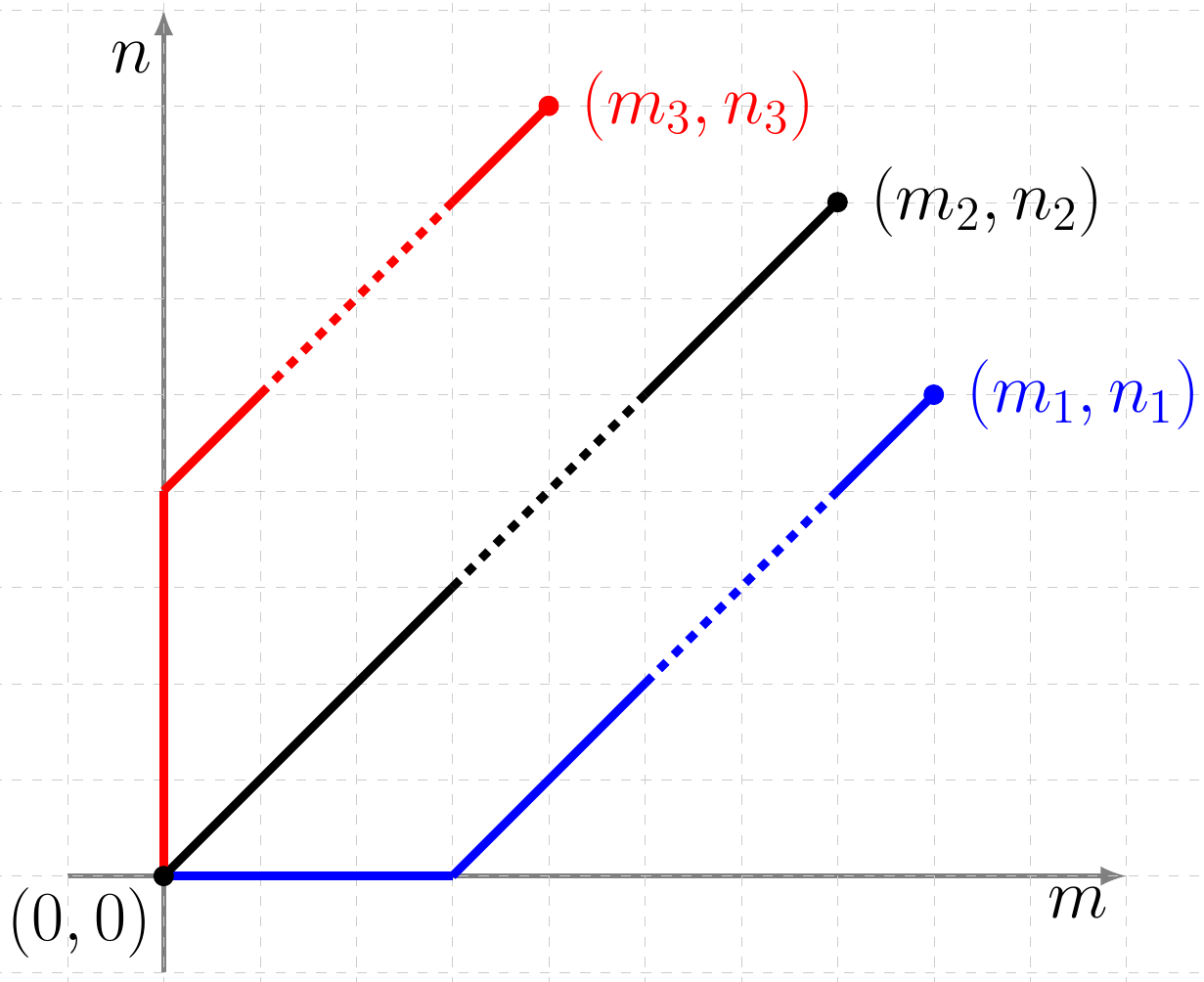}
	\end{minipage} 
	\begin{minipage}[t]{0.53\hsize}
		
		\includegraphics[width=1.0\textwidth]{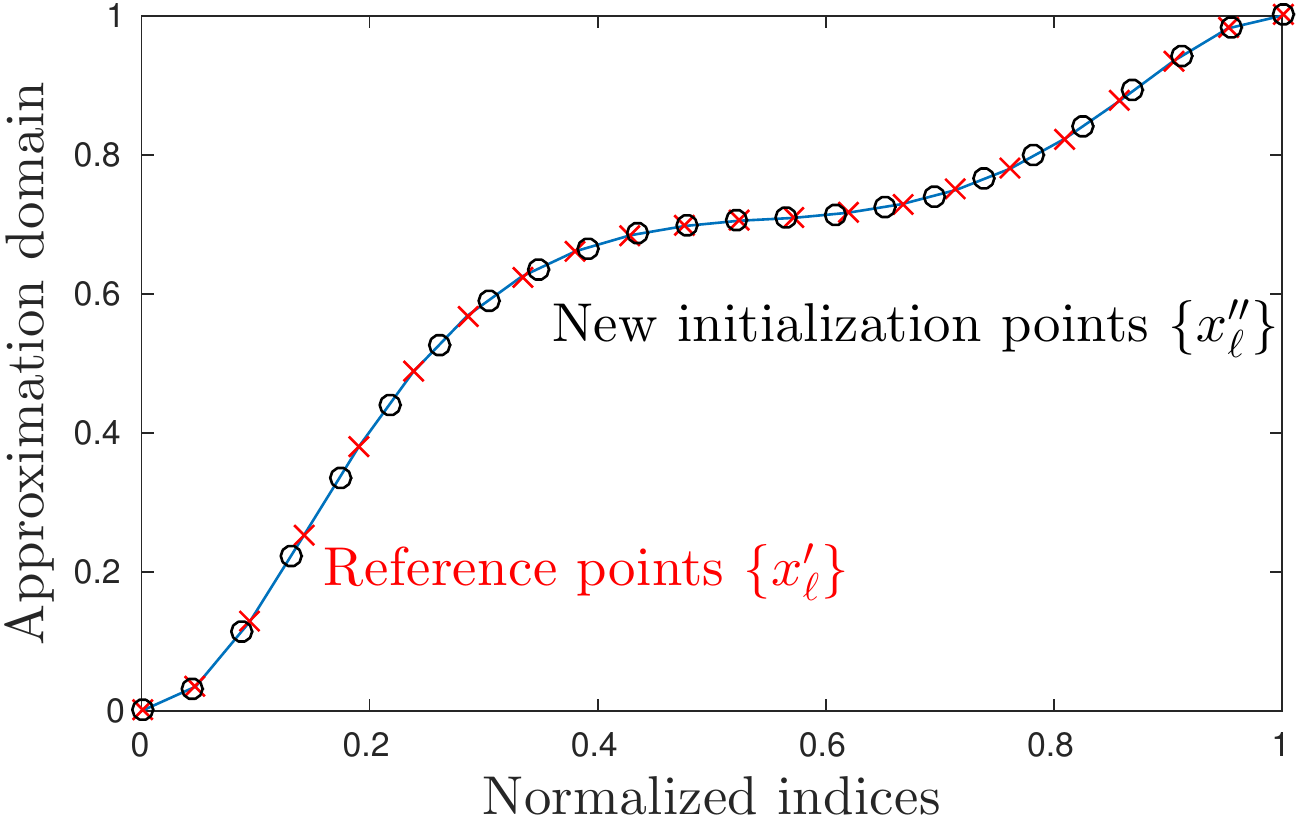}
	\end{minipage}
	
	\caption{Initialization with lower degree approximations. The left plot shows the three possible paths for updating the degrees (assuming the increment is $j=1$): $m<n$ (red), $m=n$ (black) and $m>n$ (blue). The right plot shows how initialization is done at an intermediate step. The function is $f_1$ from Table~\ref{table:remezother}, with a singularity at $x=1/\sqrt{2}$. The $y$ components of the red crosses correspond to the final references $\{x_\ell'\}$ for the $(m',n')=(10,10)$ best approximation, while the $y$ components of the black circles are the initial guess $\{x_{\ell}''\}$ for the $(m'',n'')=(11,11)$ problem, taken based on the piecewise linear fit at $\{x_\ell'\}$. Note how the $y$ components of both sets of points cluster near the singularity.}
	\label{fig:lowerdegree}
\end{figure}

\subsection{Using lower degree approximations}\label{sec:inicdf} We resort to this strategy if CF and AAA-Lawson fail to produce a sufficiently good initial guess. For functions $f$ with singularities in $[a,b]$, the reference sets $\{x_\ell\}$ corresponding to best approximations in~\eqref{eq:alternori} tend to cluster near these singularities as $m$ and $n$ increase.

%\rrrr{A constructive approach of measuring this effect, one that hopefully leads to an automatic way of determining the best type $(m,n)$ approximation (in case the CF and AAA-Lawson strategies fail to produce a initial reference set that results in convergence of the Remez algorithm) is to first construct a type $(m',n')$ best approximation with $m'\ll m$ and $n'\ll n$.}

It is sensible to expect that first computing a type $(m',n')$ best approximation to $f$ with $m'\ll m$ and $n'\ll n$ is easier (with convergence achieved if necessary with the help of CF or AAA-Lawson). We then proceed by progressively increasing the values of $m'$ and $n'$ by small increments $j$, typically $j\in\left\lbrace 1,2,4\right\rbrace$. The steps taken follow a diagonal path, as explained in Figure~\ref{fig:lowerdegree}.
Note that in addition to improving the robustness of the Remez algorithm, this strategy can help detect degeneracy; recall the discussion after~\eqref{eq:alternori}. It proves useful for many examples, including some of those shown in Section~\ref{sec:numres}: type $(n,n)$ approximations to $f(x)=|x|, x\in[-1,1]$ for $n>40$ and the $f_1,f_2$ and $f_4$ specifications in Table~\ref{table:remezother}.
 
%In more detail, at each step, we use the final reference set of the previous step to initialize our Remez code for the new problem. 
%Specifically, assuming that the increment is from $(m',n')$ to $(m'',n'')$ and $x_0',\ldots,x_{m'+n'+1}'$ are the final reference points for the $(m',n')$ instance (given in ascending order), 
%\rrrr{we construct a piecewise linear function that interpolates the cumulative distribution of the $x_k'$'s based on the index $k$. That is,} 
%we normalize the indices of the $m'+n'+2$ reference points to $[0,1]$ by $d_\ell=\frac{\ell}{m'+n'+1}, \ell=0,\ldots,m'+n'+1$ and construct a piecewise linear fitting $f_d$ at the $\{d_\ell\}$ that satisfies $f_d(d_\ell)=x_\ell'$. We take the initial guess for the $(m'',n'')$ instance as  $x_\ell''=f_d\left(\frac{\ell}{m''+n''+1}\right), \ell=0,\ldots,m''+n''+1$. An example is given in the right plot of Figure~\ref{fig:lowerdegree}. When $f$ is even/odd, we take care to enforce an appropriate symmetry in the initialization guess with respect to the middle reference point(s).

%We also remark that the approach assumes that there are no degeneracy issues at any of the steps.

%\section{Implementation details}\label{sec:impdetail}Here we discuss various details for implementing the barycentric rational Remez algorithm. 
\section{Searching for the new reference}\label{sec:nextrefs}
\begin{figure}[htbp]
	\centering
	\includegraphics[width=0.8\textwidth]{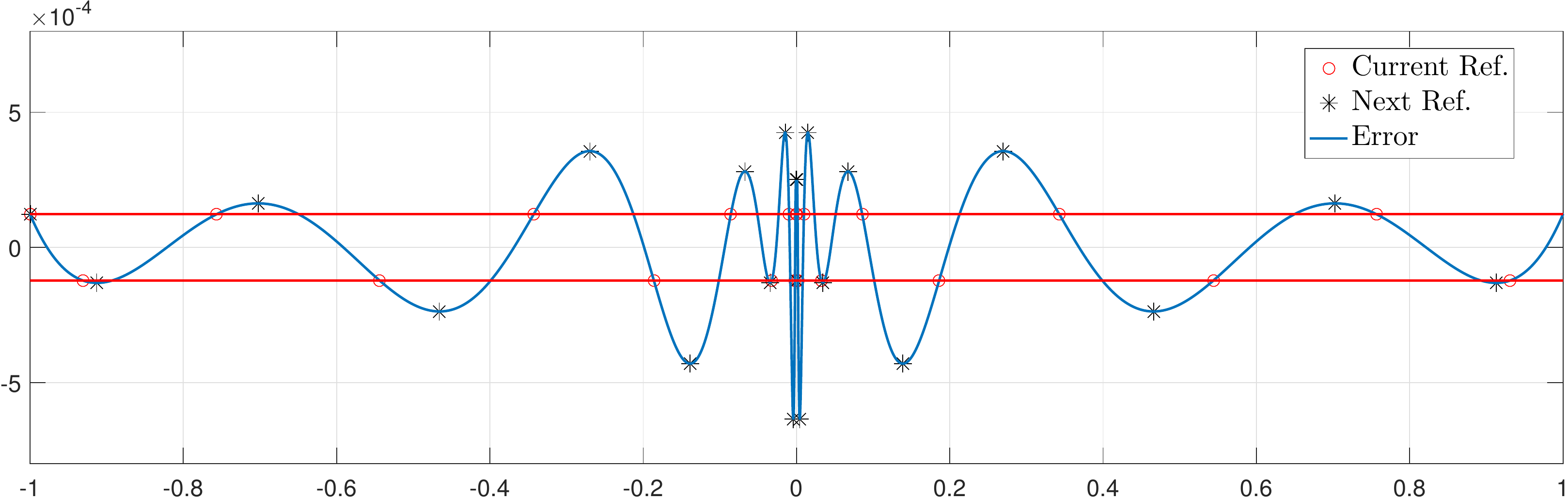}  
	\caption{
		Illustration of how a new set of reference points (black stars) is found from the current error function $e=f-r$ (blue curve). 
		Shown here is the error curve after three Remez iterations in finding the best type $(10,10)$ approximation to $f(x)=|x|$ on $[-1,1]$. We split this interval into subintervals separated by the previous reference points (red circles), and approximate $e$ on each subinterval by a low-degree polynomial. We then find the roots of its derivative. }
	\label{fig:plotfindref}
	
\end{figure}
We now turn to the updating strategy for the reference points $x_0\ldots,x_{m+n+1}$ during the Remez iterations. These are a subset of the local extrema of the error function $e(x)=f(x)-r(x)$. To find them, we 
decompose the domain $[a,b]$ into subintervals of the form $[\tilde x_\ell,\tilde x_{\ell+1}]$ (and $[a,\tilde x_0]$ and $[\tilde x_{m+n+1},b]$, if non-degenerate; here $\{\tilde x_\ell\}$ are the old reference points) and then compute Chebyshev interpolants $p_e(x)$ of $e(x)$ on each subinterval. 
In addition, if $f$ has singularities (identified by Chebfun's {\tt splitting on} functionality~\cite{pachon2010piecewise}), then we further divide the subintervals at those points.
Since $e(x)$ is then smooth and each subinterval is small, typically a low degree suffices for $p_e=\sum_{i=0}^k c_iT_i(x)$: we start with $2^3+1$ points (degree $k =8$), and resample if necessary (determined by examining the decay of the Chebyshev coefficients).
We then find the roots of $p_e'(x) = \sum_{i=1}^k ic_iU_{i-1}(x)$ 
(using the formula $T_n'(x)=nU_{n-1}(x)$) via the eigenvalues of the colleague matrix for Chebyshev polynomials of the second kind~\cite{good1961colleague}. 
Typically, one local extremum per subinterval is found, resulting in $m+n+2$ points, including the endpoints. If more extrema are found, we evaluate the values of $|e(x)|$ at those points and select those with the largest values that satisfy~\eqref{eq:newaltern}.

\section{Numerical results}\label{sec:numres}
All computations in this section were done using Chebfun's new~\texttt{minimax} command in standard IEEE double precision arithmetic.

Let us start with our core example of approximating $|x|$ on $[-1,1]$, a problem discussed in detail in~\cite[Ch.~25]{trefethen2013approximation}. For more than a century, this problem has attracted interest. The work of Bernstein and others in the 1910s led to the theorem that degree $n\geq 0$ polynomial approximations of this function can achieve at most $O(n^{-1})$ accuracy, whereas Newman in 1964 showed that rational approximations can achieve root-exponential accuracy~\cite{newman1964}. The convergence rate for best type $(n,n)$ approximations was later shown by Stahl~\cite{Stahl93} to be $E_{n,n}(|x|,[-1,1])\sim 8e^{-\pi\sqrt{n}}$.

\begin{figure}
	\centering
	\includegraphics[width=0.8\linewidth]{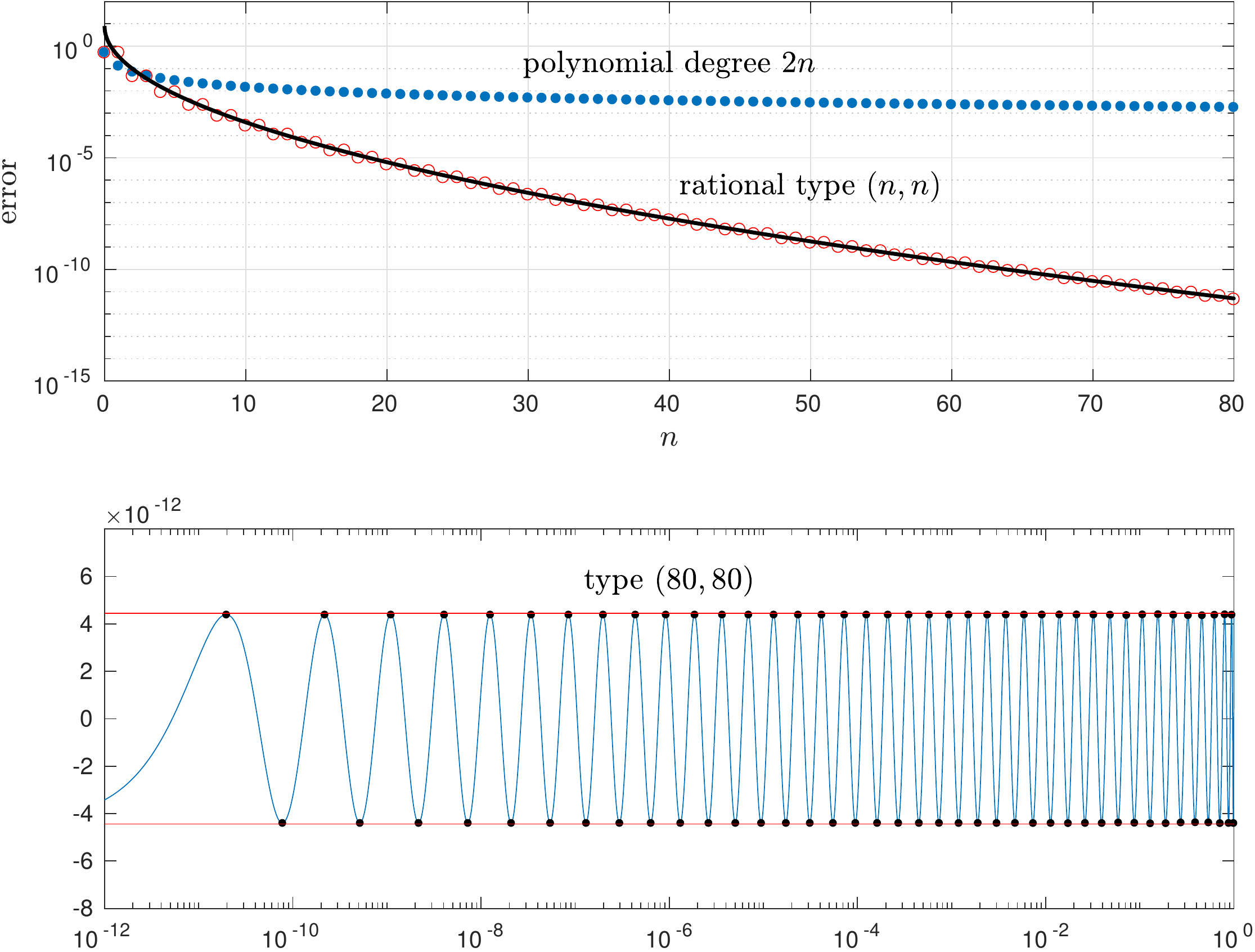}
	\caption{In the first plot, the upper dots show the best approximation errors for the degree $2n$ best polynomial approximations of $|x|$ on $[-1,1]$, while the lower ones correspond to the best type $(n,n)$ rational approximations, superimposed on the asymptotic formula from~\cite{Stahl93}. The bottom plot shows the minimax error curve for the type $(80,80)$ best approximation to $|x|$. 
Note that the horizontal axis has a log scale: the alternant ranges over 11 orders of magnitude. The positive part of the domain $[-1,1]$ is shown (by symmetry the other half is essentially the same).}
% the abscissa axis is a log scale of the $[0,1]$ half of the domain $[-1,1]$.
	\label{fig:absErrors}
\end{figure}

This result had in fact been conjectured by Varga, Ruttan and Carpenter~\cite{VargaEtAl93} based on a specialized multiple precision (200 decimal digits) implementation of the Remez algorithm. Their computations were performed on the square root function, using the fact that $E_{2n,2n}(|x|,[-1,1])=E_{n,n}(\sqrt{x},[0,1])$, as follows from symmetry. They went up to $n=40$. In both settings,  the equioscillation points cluster exponentially around $x=0$ (see second plot of Figure~\ref{fig:absErrors}), making it extremely difficult to compute best approximations. Our barycentric Remez algorithm in double precision arithmetic is able to match their performance, in the sense that we obtain the type $(80,80)$ best approximation to $|x|$ in less than 15 seconds on a desktop machine. The results are showcased in Figure~\ref{fig:absErrors}, where our levelled error computation for the type $(80,80)$ approximation (value $4.39\ldots \times 10^{-12}$) matches the corresponding error of~\cite[Table 1]{VargaEtAl93} to two significant digits, even though the floating point precision is no better than $10^{-16}$. 

Running the other non-barycentric codes (Maple's \texttt{numapprox[minimax]}, Mathematica's \texttt{MiniMaxApproximation} (which requires $f$ to be analytic on $[a,b]$), and Chebfun's previous {\tt remez}) on the same example resulted in failures at very small values of $n$ (all for $n\leq 8$). 

\begin{table}[h!]
	\small
	\centering
	\caption{
Best approximation to five difficult functions by the barycentric rational Remez algorithm. $f_1''$ is discontinuous at $x=1/\sqrt{2}$, $f_2'$ is discontinuous at $x=0$, $f_3'$ is unbounded as $x\to 0$, $f_4$ has two sharp peaks at $x=\pm 0.6$, and $f_5$ has a logarithmic singularity at $x=0$.\small}
	\label{table:remezother}
	\begin{tabular}{lcccc}
		\hline 
		$i$ & $f_i$ & $[a,b]$ & $(m,n)$ & $\left\Vert f-r^*\right\Vert_{\infty}$ \T\B \\  \hline 
\\
&&&& \vspace{-6mm}\\ \vspace{2mm} % extra space
		1 & $\begin{cases}
		x^2, & x<\frac{1}{\sqrt{2}} \\
		-x^2+2\sqrt{2}x-1, & \frac{1}{\sqrt{2}}\leq x 
		\end{cases}$ & $[0,1]$ & $(22,22)$        & $2.439\times 10^{-9}$                                       \\ \vspace{2mm}

		2 & $|x|\sqrt{|x|}$ & $[-0.7,2]$ & $(17,71)$         & $4.371\times 10^{-8}$  \\  
\vspace{2mm}
		3 & $x^3+\dfrac{\sqrt[3]{x}e^{-x^2}}{8}$       & $[-0.2,0.5]$ & $(45,23)$ & $2.505\times 10^{-5}$                                       \\  \medskip
		4 & $\dfrac{100\pi(x^2-0.36)}{\sinh(100\pi(x^2-0.36))}$       & $[-1,1]$        & $(38,38)$        &  $1.780\times 10^{-12}$                                      \\  \vspace{2mm}
		5 & $-\dfrac{1}{\log|x|}$  & $[-0.1,0.1]$ & $(8,8)$ & $1.52\times 10^{-2}$                                   \\ \hline
	\end{tabular}
\end{table}
\normalsize 

\begin{figure}[htbp]
	\centering
	\includegraphics[width=1.0\linewidth]{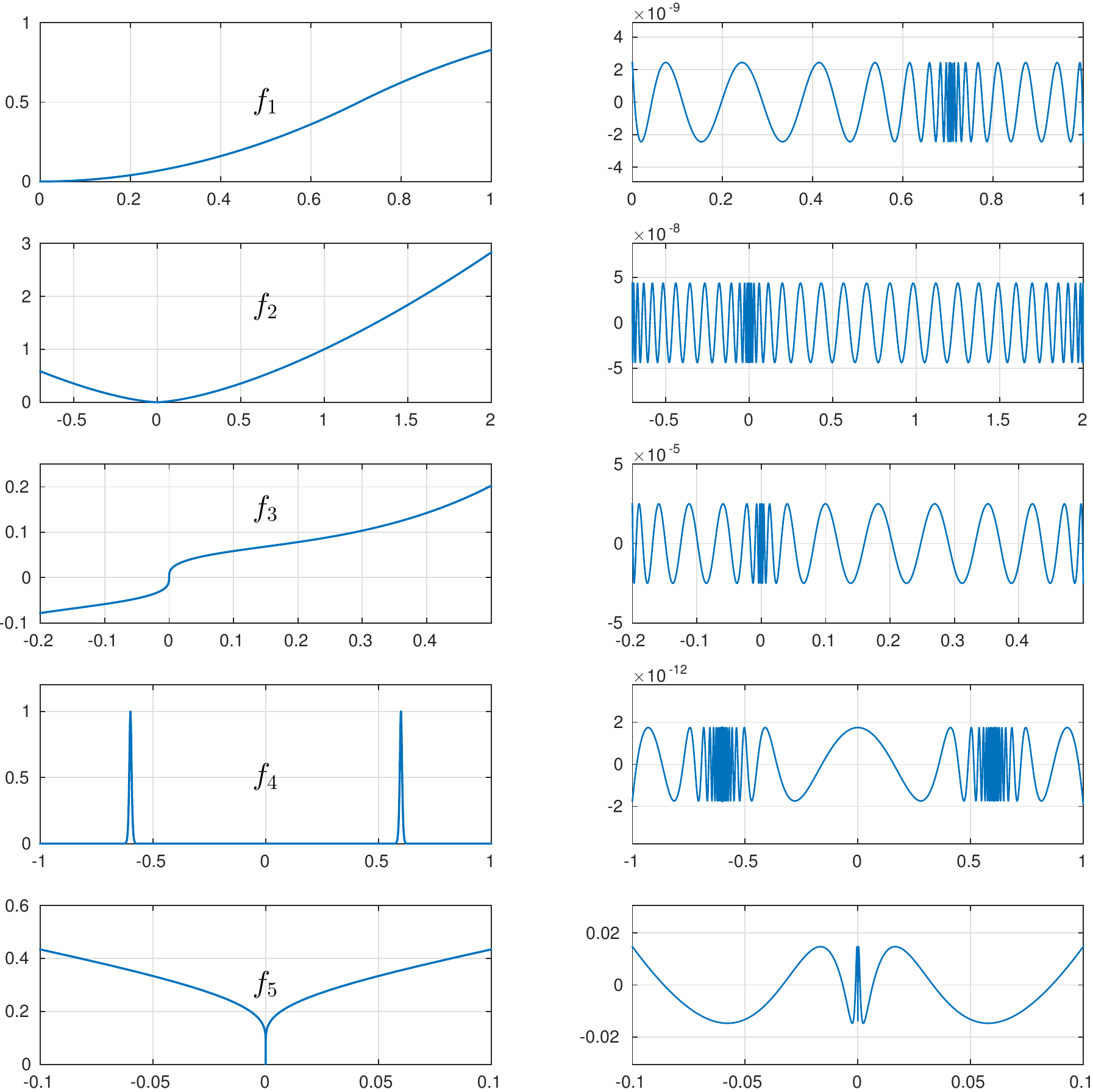}
	\caption{Error curves for the best rational approximations 
of  Table~\ref{table:remezother}.}
	\label{fig:remezErrors}
\end{figure}

The robustness of our algorithm is also illustrated by the examples of Table~\ref{table:remezother} and Figure~\ref{fig:remezErrors}, which is a highlight of the paper. Computing these five approximations takes in total less than 50 seconds with~\texttt{minimax}. Example $f_4$ is taken from~\cite[\S 5]{van2010computing}, while $f_5$ is inspired by~\cite{pushnitski2017best}. The difficulty of approximating $f_5$ is even more pronounced than for $|x|$, since best type $(n,n)$ approximations to $f_5$ offer at most $O(n^{-1})$ accuracy (a stark contrast to the root-exponential behavior of $E_{n,n}(|x|,[-1,1])$) and the reference points cluster even more strongly, quickly falling below machine precision.

In Figures~\ref{fig:sqrtrelerr} and~\ref{fig:zolotarev}, we further illustrate \texttt{minimax} and its weighted variant, by revisiting some classical problems in rational approximation: the Zolotarev problems~\cite[Ch.~9]{elementselliptic}. 
Among other questions, 
Zolotarev asked what are the best rational approximants to the sign function (on the union of intervals $[-b,-a]\cup [a,b]$ for scalars $0<a<b$) and the 
%$1/\sqrt{x}$ function (in the relative sense, i.e., minimizing $\|1-\frac{r}{\sqrt{x}}\|_{\infty}$) on $[1,K^2]$. 
$\sqrt{x}$ function (in the relative sense, i.e.,~minimizing $\|1-r/\sqrt{x}\|_{\infty}$) on $[1/b^2,1/a^2]$. 
%the union of intervals . 
%Another problem he considers is  the best rational approximant to $\sqrt{x}$ in the relative sense on $[1,\sqrt{K}]$, that is, the minimizer for 
 Zolotarev proved these problems are mathematically equivalent
through the identity $\mbox{sign}(x) = x\sqrt{1/x^2}$: 
if $r$ is the type $(m,m)$ best approximant to $\sqrt{x}$ on $[1/b^2,1/a^2]$, then
% taking $x=1/y^2$, 
$\mbox{sign}(x)-xr(1/x^2)$ is found to equioscillate at $4m+4$ points on 
$[-b,-a]\cup [a,b]$, so $xr(1/x^2)$ is the best approximant to $\mbox{sign}(x)$ of type $(2m+1,2m)$ on $[-b,-a]\cup [a,b]$. 
%the function $x\sqrt{1/x^2}$ is of type $(2m+1,2m)$, and 
Furthermore, Zolotarev gave explicit solutions involving Jacobi's elliptic functions. These rational functions have the remarkable property of preserving optimality under appropriate composition~\cite{nakfrsirev}. 
In Figure~\ref{fig:sqrtrelerr} we compute the best relative error approximant of type $(m,m)$ to $\sqrt{x}$ 
using the weighted variant of our rational Remez algorithm. We then compute  $xr(1/x^2)$, the type $(2m+1,2m)$ best approximant to the sign function. The error function is shown in Figure~\ref{fig:zolotarev}, confirming Zolotarev's results.

We emphasize that the  examples presented in this
section are extraordinarily challenging, far beyond the capabilities
of most codes for minimax approximation.  Chebfun~\texttt{minimax} not only
solves them but does so quickly.
For smoother functions such as analytic functions (with singularities, if any, lying far from the interval), we find that \texttt{minimax} usually easily computes $r^*$ so long as $\|f-r^*\|_\infty$ is a digit or two larger than $u\|f\|_\infty$.
%until $\|f-r^*\|_\infty=O(u)$. 

\begin{figure}[htbp]
	\centering
	\includegraphics[width=\linewidth]{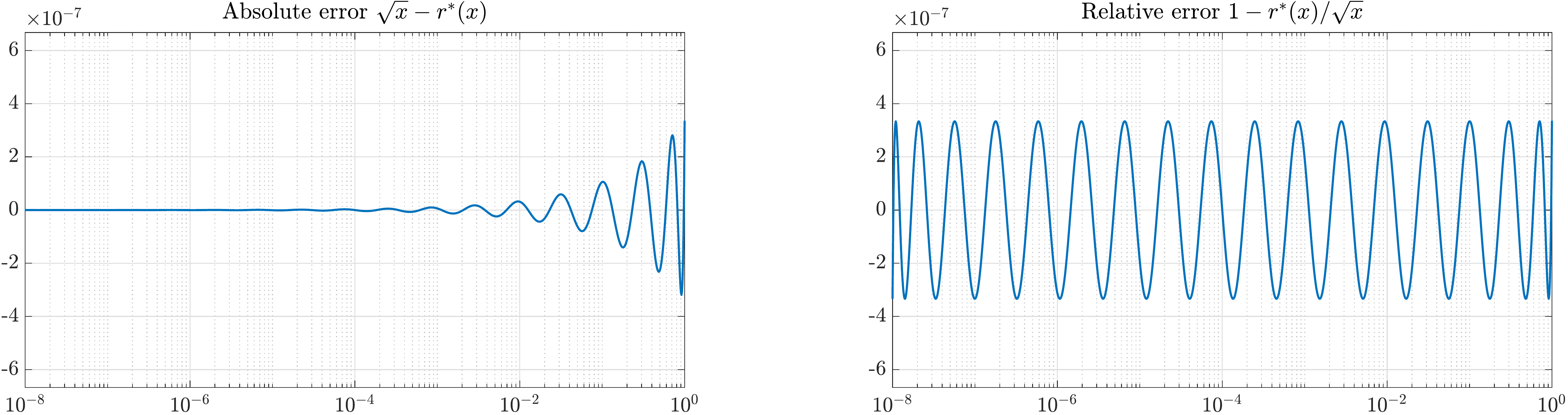}
	\caption{
%This example shows the 
Result of the weighted version of our barycentric Remez algorithm for the function 
%$f(x)=\sqrt{x},x\in[10^{-5},10^5]$ 
$f(x)=\sqrt{x},x\in[10^{-8},1]$ 
with $w(x)=1/\sqrt{x}$ and a type 
%$(16,16)$ 
$(17,17)$
rational approximation. We plot the absolute error curve on the left, while the relative error (right), matching our choice of $w$, gives an expected equioscillating curve. This is Zolotarev's third problem.
}\label{fig:sqrtrelerr}
\end{figure}

\begin{figure}[htbp]
	\centering
	\includegraphics[width=\linewidth]{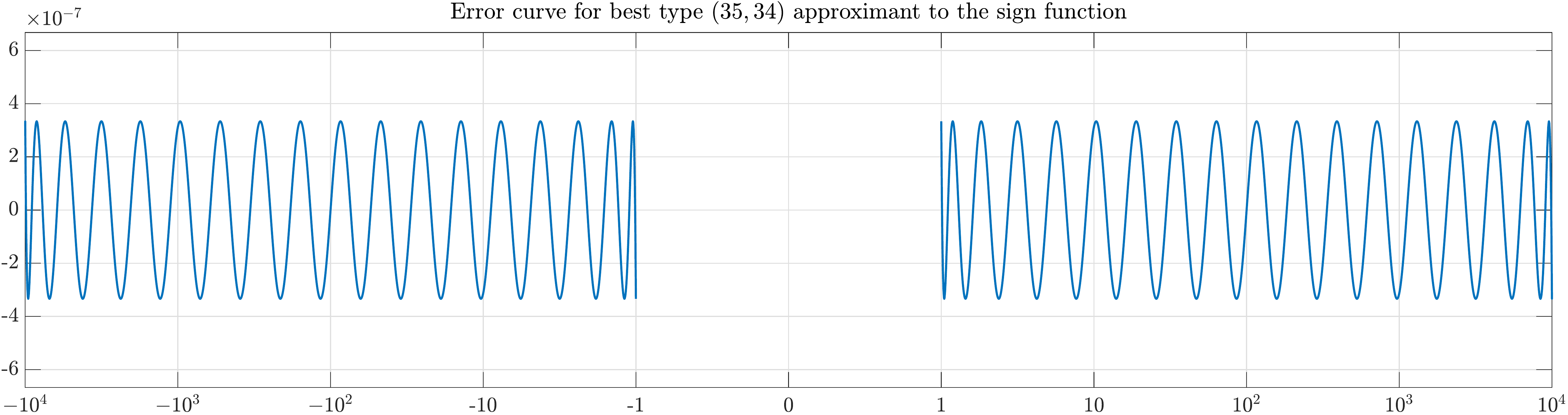}
	\caption{The error in type 
%$(40,40)$ 
$(35,34)$
best approximation to the sign function on $[-10^4,-1]\cup[1,10^4]$, 
 computed via $xr(1/x^2)$, where $r(x)\approx \sqrt{x}$ as obtained in Figure~\ref{fig:sqrtrelerr}. This 
is %corresponds to an instance of
 Zolotarev's fourth problem. 
%a type $(19,20)$ best weighted approximation to $f(x)=1/\sqrt{x},x\in[1,10^8]$ and $w(x)=\sqrt{x}$. 
%Notice that $f$ is also an example of a Markov function.
}\label{fig:zolotarev}
\end{figure}
%\begin{figure}[htbp]
%	\centering
%	\includegraphics[width=\linewidth]{figures/relerr_invsqrt.pdf}
%	\caption{Type $(19,20)$ best relative error approximation to $f(x)=1/\sqrt{x},x\in[10^{-8},1]$.}\label{fig:invsqrt}
%\end{figure}

\section{AAA-Lawson algorithm}\label{sec:iniAAA}
%\section{Initial reference points by aaa-Lawson}\label{sec:aaalawson}
Here we describe a new
algorithm for rational approximation that we call the AAA-Lawson
algorithm; in practice we recommend this for computing an initial
guess for the Remez iteration.  It applies on a finite, discrete
set rather than the continuous interval $[a,b]$ as in~\eqref{eq:mainproblem}.
%Here we describe the AAA-Lawson algorithm, which can be regarded as an independent algorithm for the rational minimax problem, but on a finite, discrete set rather than the continuous interval $[a,b]$ as in~\eqref{eq:mainproblem}. 
Specifically, we consider the problem
\begin{equation}  \label{eq:mainproblemdiscrete}
\minimize_{r\in\mathcal{R}_{m,n}}\|f(Z)-r(Z)\|_{\infty},
\end{equation}
where $Z=\{z_1,\ldots,z_M\}$ is a set of distinct points (\emph{sample points}) in $[a,b]$. 
The number $M$ is usually large, $\textnormal{e.g.~}10^5$, and in
particular much bigger than $m$ and $n$.
%The idea is that 
The idea is that the solution for the discrete problem~\eqref{eq:mainproblemdiscrete} should converge to the continuous one~\eqref{eq:mainproblem} if we discretize the interval densely enough. 

AAA-Lawson proceeds as follows: 
%Our strategy for finding the initial reference points is as follows:
\medskip
\begin{enumerate}
	\item Use the AAA algorithm to find an approximant~\eqref{eq:rzbary}, in particular the support points $\left\lbrace t_k\right\rbrace$ for a rational approximation $r$ to $f$. This step is not tied to a particular norm.
	\item Use a variant of Lawson's algorithm to obtain a refined (near-best) rational approximant in the $\ell_\infty$ norm. 
\end{enumerate}

Below we first review the AAA algorithm, introduced in~\cite{aaapreprint}, then the Lawson algorithm, and then we present the AAA-Lawson combination.
\subsection{The AAA algorithm}\label{sec:aaa}
Given a function $f$ and sample points $Z\in\mathbb{C}^M$, 
the AAA algorithm finds a rational approximant of type $(n,n)$ represented 
as in~\eqref{eq:rzbary} by $r(z)=\widetilde N(z)/\widetilde D(z):=\sum_{k=0}^n f(t_k)\beta_k(z-t_k)^{-1}\big/\sum_{k=0}^n \beta_k(z-t_k)^{-1}$. 
%\begin{equation}  \label{eq:pqcbary}r(x):=\frac{N(x)}{D(x)}=\frac{\sum_{i=0}^n \frac{f(t_i)\beta_i}{x-t_i}}{\sum_{i=0}^n \frac{\beta_i}{x-t_i}}. \end{equation}
Here, the support points $\{t_k\}$ are a subset of $Z$ chosen in an adaptive, greedy manner so as to improve the approximation  as we increase $n$, exploiting the interpolatory property $\widetilde N(t_k)/\widetilde D(t_k)=f(t_k)$ for all $k$ (unless $\beta_k=0$).
%While the In order to represent a (near-)best rational approximant, discuss a non-interpolatory variant of AAA, which will be needed for implementing Lawson's algorithm. 
%We describe the second step below. Before doing so, we 
%discuss a non-interpolatory variant of AAA, which will be needed for implementing Lawson's algorithm. 
%An important aspect of the approximant $r$ is the interpolation property %$\alpha_i=f(t_i)\beta_i$
%$r(t_k)=f(t_k)$ for all $k$ (unless $\beta_k=0$).
%Indeed, 
AAA takes only $\beta_k$ as the unknowns, which are found by solving a %least-squares fitting with respect to the 
linearized least-squares problem of the form $\minimize_{\|\beta\|_2=1}\|f\widetilde D-\widetilde N\|_{\Zt}$, where the subscript $\Zt$ denotes the discrete $2$-norm at points $\Zt:=Z\setminus\left\lbrace t_0,\ldots,t_n\right\rbrace$. For details, see~\cite{aaapreprint}. 

\paragraph{Noninterpolatory  AAA}\label{sec:aaanoninterp}
%The rational approximant in standard AAA is represented as $\widetilde N(x)/\widetilde D(x)$ in AAA to take advantage of the interpolation property, 
As we discussed in Section~\ref{sec:barybasic}, the representation $\widetilde N(z)/\widetilde D(z)$  is unsuitable when the goal is to represent $r^*$: it is necessary to use the representation $r(z)=N(z)/D(z)=\sum_{k=0}^n\alpha_k(z-t_k)^{-1}\big/\sum_{k=0}^n\beta_k(z-t_k)^{-1}$ as in~\eqref{eq:baryr}. 
This leads to a noninterpolatory variant of AAA, discussed briefly in~\cite[Section~10]{aaapreprint}. The resulting least-squares problem 
$\minimize_{
\|\alpha\|^2_2+\|\beta\|^2_2=1}\|fD-N\|_{\Zt}$ has unknowns $\alpha$ and $\beta$. Written in matrix form, it takes the form
\begin{equation}  \label{eq:aaalawls}
\minimize_{
\|\alpha\|^2_2+\|\beta\|^2_2=1
%\left\|	\begin{smallmatrix}	\alpha\\\beta	\end{smallmatrix}\right\|_2=1
}
\left\|
\begin{bmatrix}
C& -FC
\end{bmatrix}
\begin{bmatrix}
\alpha \\\beta
\end{bmatrix}
\right\|_2, 
\end{equation}
where $F=\mbox{diag}(f(\Zt))$, and $C_{\ell,k}=1/(z_\ell-t_k)$ is the Cauchy (basis) matrix as in~\eqref{eq:geneigcauchy}, but with rows corresponding to $z_\ell\in \{t_0,\ldots,t_n\}$ removed.
We take the same support points $\{t_k\}$ as in AAA. 
We solve~\eqref{eq:aaalawls} by computing the SVD of the matrix 
$   \begin{bmatrix}      C& -FC    \end{bmatrix} $ 
and finding the right singular vector $v=\big[
\begin{smallmatrix}
\alpha\\ \beta  
\end{smallmatrix}
\big]\in\mathbb{R}^{2n+2}
$ corresponding to the smallest singular value. As in Section~\ref{sec:nondiag},  the case $m\neq n$ also uses the projection matrices $P_m,P_n$.

%, but $(\alpha_k)_{k=0}^{n}$ are added as unknowns. 

%\begin{equation}  \label{eq:baryratgen}r(x):=\frac{N(x)}{D(x)}=\frac{\sum_{i=0}^n \frac{\alpha_i}{x-t_i}}{\sum_{i=0}^n \frac{\beta_i}{x-t_i}}. \end{equation}

\subsection{Lawson's algorithm}\label{sec:lawson}

Lawson's algorithm~\cite{lawsonthesis} computes the best polynomial (linear) approximation based on an iteratively reweighted least-squares process. 
During the iteration, a set of weights is updated according to the residual of the previous solution. 
%It solves the minimax problem by solving a sequence of iteratively reweighted least-squares problems, During which a set of weights are updated according to the residual of the previous solution. 

Specifically, suppose that $f$ is to be approximated on $Z=\{z_1,\ldots,z_M\}$ in a linear subspace $\mbox{span}(g_i)_{i=0}^n$. With an initial set of weights $\left\lbrace w_j\right\rbrace_{j=1}^M$
%\rr{SF: what is $N$, confusion with the barycentric denominator of $r$!} 
such that $w_j\geq 0$ and $\sum_{j=1}^M w_j=1$, 
%\rrrr{where $M$ is the number of discretization points on $[a,b]$,} 
one solves (using a standard solver) the weighted least-squares problem 
\begin{equation}
\label{eq:minfcg}
\minimize_{c_0,\ldots,c_n}\|f-\sum_{i=0}^nc_ig_i\|_{w} = \sqrt{\sum_{j=1}^{M} w_j(f(Z_j)-\sum_{i=0}^nc_{i}g_i(Z_j))^2}  , 
\end{equation}
and computes the residual 
$r_j = f(Z_j)-\sum_{i=0}^nc_ig_i(Z_j)$. The weights are then updated by $w_j:=w_j|r_j|$, followed by the re-normalization $w_j:=w_j/\sum_{i=1}^Mw_i$. 
Iterating this process is known to converge linearly to the best polynomial approximant (with nontrivial convergence analysis~\cite{cline1972rate}), and an acceleration technique is presented in \cite{ellacott1976linear}. 
\subsection{AAA-Lawson}\label{sec:aaalawson}
We now propose a rational variant of Lawson's algorithm. 
(A similar attempt was made in~\cite[\S~6.5]{cooper2007rational}, 
though the formulation there is not the same: most notably, adjusting the exponent $\gamma$ as done below appears to improve robustness significantly.)
%without adjusting the exponent $\gamma$ as done below, \rr{and with slightly different least-squares problem; we observed through experiments that the version presented below converges more often}). 
%\rrrr{Here is the algorithm description.} 
The idea is to incorporate Lawson's %iteratively reweighted least-squares
approach into noninterpolatory AAA, replacing~\eqref{eq:minfcg} with a weighted version of~\eqref{eq:aaalawls}, and updating the weights as in Lawson. 

%Having computed an AAA approximant, and
Specifically, given an initial set of weights $w\in\mathbb{R}^{M-(\max(m,n)+1)}$, usually all ones, and initializing the \emph{Lawson exponent} $\gamma=1$,  
we proceed as follows:
\medskip
\begin{enumerate}
	\item[1.] Solve 
	the weighted linear least-squares problem
	\begin{equation}  \label{eq:fd-naaa}
	\minimize_{\|\alpha\|^2_2+\|\beta\|^2_2=1}\|f(\Zt)D(\Zt)-N(\Zt)\|_w,  
	\end{equation}
	via the SVD
	of the matrix $   \mbox{diag}(\sqrt{w})\begin{bmatrix}      C& -FC    \end{bmatrix} $ (recall~\eqref{eq:aaalawls}).
	If the resulting $\left\Vert f(Z)-N(Z)/D(Z)\right\Vert_\infty$ is not smaller than before, then set $\gamma:=\gamma/2$. 
	%update $D_{\mbox{old}}\leftarrow D$. Iterate $k_1$ times. 
	\medskip
	\item[2.] 
	%Set $w(s)=1/N$, then iteratively solve $\|\sqrt{w(s)}D_{\mbox{old}}(s)(f(s)D(s)-N(s))\|_\infty\rightarrow \minimize$, 
	Update $w$ by 
	\begin{equation}
	\label{eq:wupdate}
	w_j\leftarrow w_j\left|f(Z_j)-\frac{N(Z_j)}{D(Z_j)}\right|^\gamma  , \quad 
	\forall j,
	\quad \mbox{then}\quad w_j:=\frac{w_j}{\sum_i w_i} 
	\end{equation}
	and return to step 1. 
	%\rr{SF: what is $M$?}
	%Iterate $k_2$ times. 
\end{enumerate}
\medskip
%Throughout the iterations, 
Note the exponent $\gamma$ in~\eqref{eq:wupdate}. In the linear case, this is  $\gamma=1$. In the rational (nonlinear) case, 
for which experiments suggest convergence is a delicate issue,
we have found that taking $\gamma$ to be smaller makes the algorithm much more robust. 
%We have included the square root as it  made the Lawson update more robust for challenging problems. 
We repeat the steps until $w$ undergoes small changes, e.g.~$10^{-3}$, or a maximum number of iterations (e.g.~30) is reached. 

We refer to this algorithm as AAA-Lawson. 
Each iteration is computed by an SVD of an $(M-\max(m,n)-1)\times (m+n+2)$ matrix, so 
the cost for $k$ iterations is $O(kM(m+n)^2)$.
Convergence analysis appears to be 
highly nontrivial and is out of our scope. %\rr{; indeed the differential correction algorithm (Section~\ref{sec:dc}) is much more robust}. 
We simply note here that if equioscillation 
of $f-N/D$ is 
achieved at $m+n+2$ points in $Z_*\subset Z$, then by defining $w^*$ as $w_j^*=1/\sqrt{|D(Z_j)|}$ for 
$j\in Z_*$ and $0$ otherwise, we see that $w^*/\sum w^*$ 
(together with $N^*/D^*=r^*$, the solution of~\eqref{eq:mainproblem})
is a fixed point of the iteration. 

%In our algorithm we employ one more 
%\rr{SF: same remark as before regarding $N$.} 
\subsection{Experiments with AAA-Lawson}
Figure~\ref{fig:absxlawson} compares AAA and AAA-Lawson
(run for ten Lawson steps) for type (10,10) and (20,20) approximation
of $f(x) = |x|$.
%We take $k_1=10$ and $k_2=40$, and terminate iterating if $\|w-w_{\mbox{old}}\|_\infty<1e-3$. 
The sample points are $10^4$ equispaced points on $[-1,1]$. 
Observe that the Lawson update 
significantly reduces the error and brings the error curve close
to equioscillation. 
\begin{figure}[htpb]
	\begin{minipage}[t]{0.49\hsize}
		\includegraphics[height=50mm]{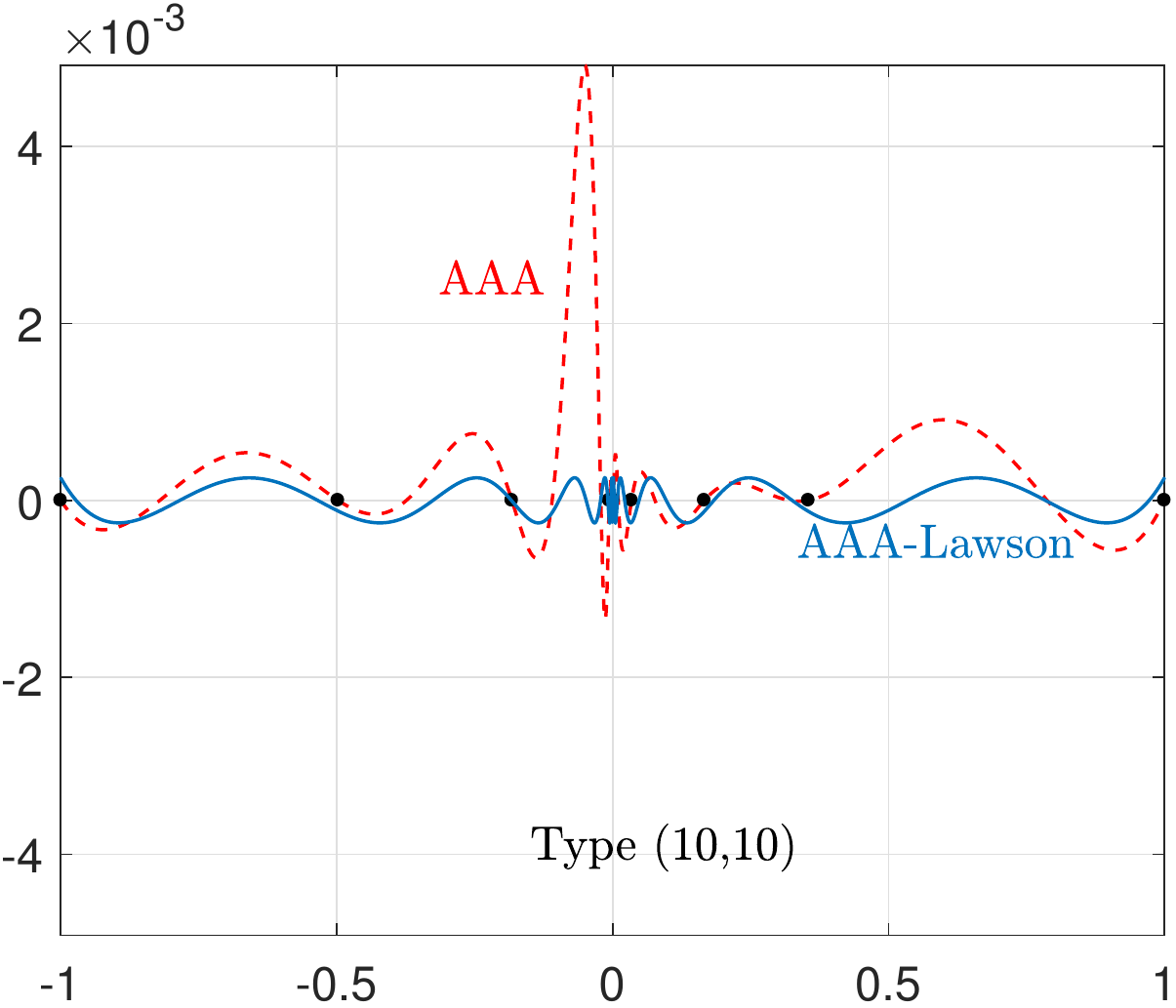}
	\end{minipage} 
	\begin{minipage}[t]{0.49\hsize}
		\includegraphics[height=50mm]{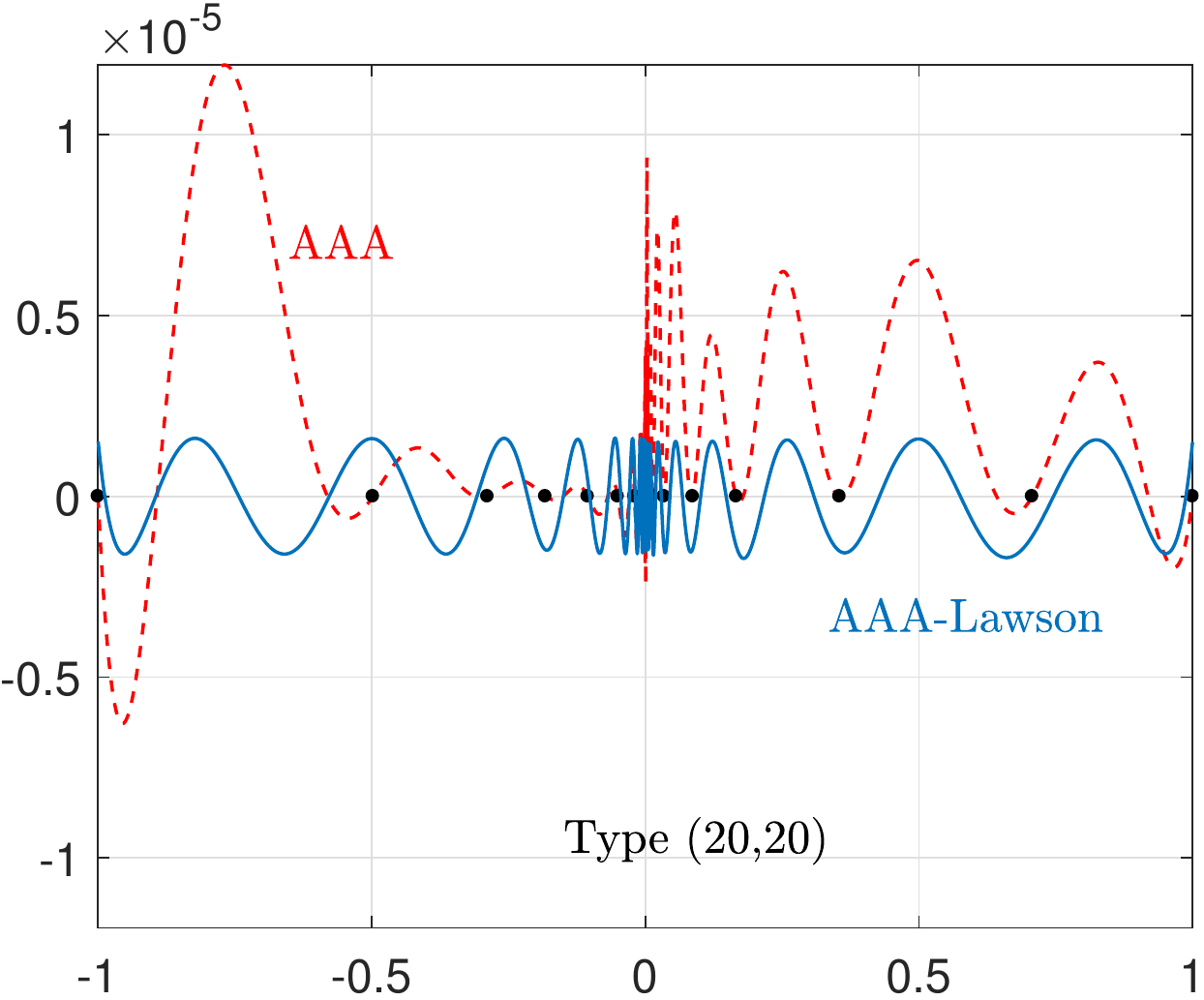}
	\end{minipage}
	\caption{
		Error of rational approximants to $f(x)=|x|$ by the AAA and AAA-Lawson algorithms. 
		%of type $(10,10)$ (left) and $(20,20)$ (right) before (red) and after (black) ten steps of Lawson refinement.
		The black dots are the support points. They are also 
		interpolation points for AAA, but not for AAA-Lawson.  
	}
	\label{fig:absxlawson}
\end{figure}

AAA-Lawson is a new algorithm for rational minimax approximation. 
% that can efficiently find an approximate (to low accuracy). 
However, we do not recommend it as a practical means to obtain $r^*$ over the classical Remez or differential correction algorithms. The reason is that its convergence is far from understood, and even when it does converge, the rate is slow (linear at best). We illustrate this in Figure~\ref{fig:aaavsremez}. In our Remez algorithm context, we take a small number (say 10) of AAA-Lawson steps to obtain a set of initial reference points, thereby taking advantage of the initial stage of the AAA-Lawson convergence. 
%\rr{thereby taking advantage mostly of the initial stage of the AAA-Lawson convergence. SF: this last part in red is not very clear.} 
%Since experiments suggest the convergence of the algorithm is no better than linear, we use the outcome as the initial approximant for the Remez algorithm (which usually converges quadratically~\cite{curtisosborne1966quadratic}) to further refine the solution to working accuracy. 

\begin{figure}[htbp]
	\centering
	\includegraphics[width=0.4\textwidth]{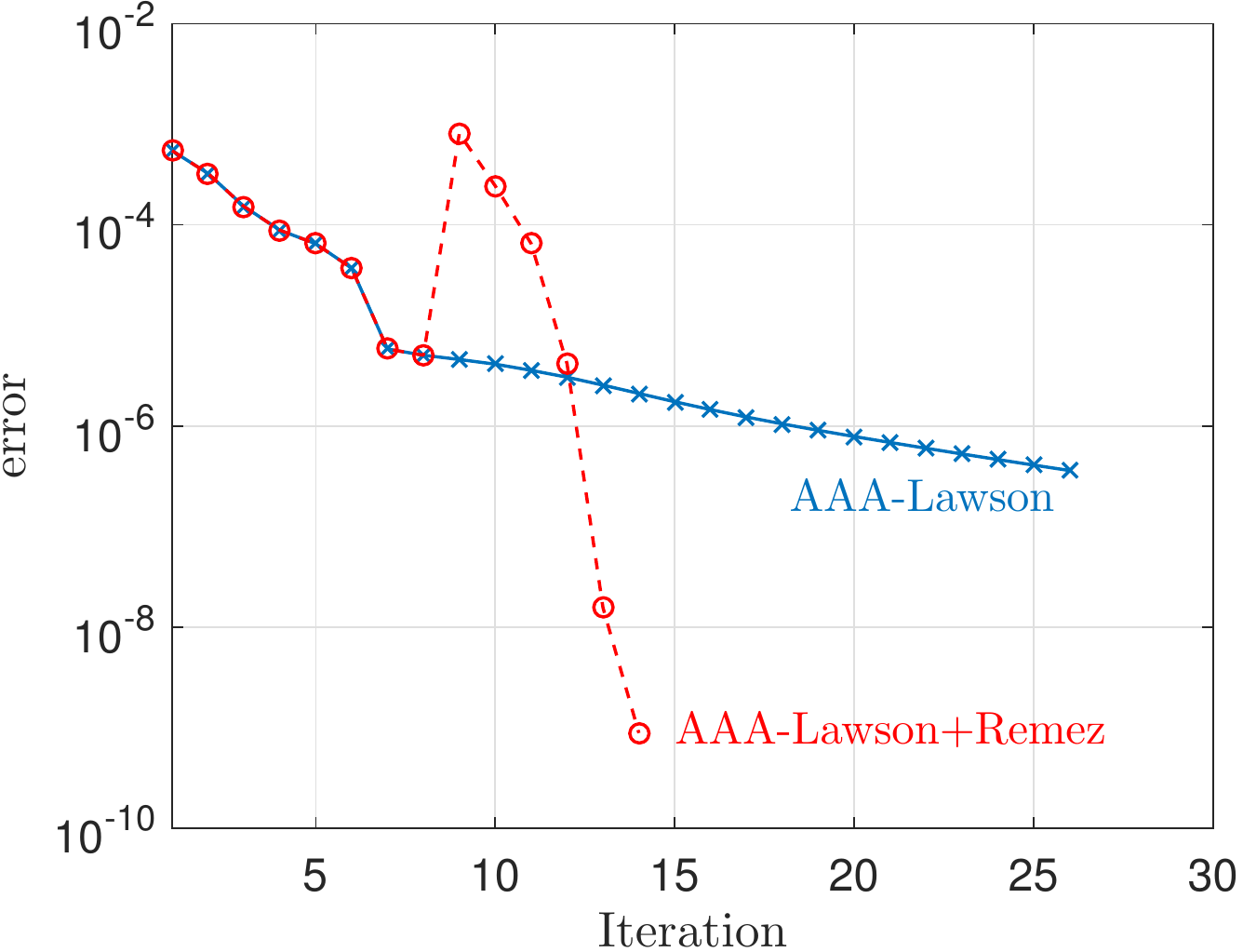}  
	\caption{Convergence of AAA-Lawson alone and AAA-Lawson followed by Remez,
		 for $f(x)=|x|$, $m=n=10$.
		The error is measured by $\|r^*-r_k\|_{\infty}$, where $r_k$ is the $k$th iterate. AAA-Lawson converges linearly, whereas
		Remez converges quadratically. 
	}
	\label{fig:aaavsremez}
\end{figure}

We note that other approaches 
for rational approximation are available, which can be used for initializing Remez. These include the Loewner approach presented in~\cite{mayo2007framework} and RKFIT~\cite{berljafa2017rkfit}. In particular, the Loewner approach is well suited when approximating smooth functions (and sometimes non-smooth functions like $f_4$~\cite{Karachalios}), %in which case it usually
often achieving an error of the same order of magnitude as the best approximation. Our experiments suggest that AAA-Lawson is at least as efficient and robust as these alternatives.

%This allows us to find the extrema of $e$ by looking at the eigenvalues of Chebyshev companion matrices (so-called \emph{colleague} matrices). 
%\rr{TODO: this a rough idea of what we do, but this can be expanded upon; we found this approach to be the most robust so far, provided we can compute $e(x)$ accurately at each iteration}.

%Low-degree Chebfun construction between reference points..

\subsection{Adaptive choice of support points}\label{sec:aaasupport}
At an early stage of the AAA-Lawson iteration, we usually do not have the correct number ($m+n+2$) of reference (oscillation) points in the error curve. 
Therefore, choosing the support points $\{t_k\}$ as in~\eqref{eq:supportchoose} is not an option. 
Instead, we use the same support points chosen by the AAA algorithm, which is typically a good set. 
Once convergence sets in and the error curve of the AAA-Lawson iterates has at least $m+n+2$ alternation points, we can switch to the adaptive choice 
\eqref{eq:supportchoose} as in Remez. We note, however, that adaptively changing the support points may further complicate the convergence, since it changes the linear least-squares problem~\eqref{eq:fd-naaa}.%, as it becomes nontrivial how to update the weights $w$ in~\eqref{eq:wupdate}. 

\subsection{Adaptive choice of the sample points}
%\rr{For the use of AAA-Lawson in the context of obtaining initial reference points for Remez,} it remains to discuss the
%\rr{Unlike Remez, AAA-Lawson (and the DC algorithm) is strictly speaking an algorithm for the \emph{dicrete} $\ell_\infty$ minimax problem, rather than the continuous problem~\eqref{eq:mainproblem}. The idea is that the discrete solution should converge to the continuous one if we discretize the interval appropriately and densely enough. }
For solving the continuous problem~\eqref{eq:mainproblem}, we take the sample point set $Z$ to be $M$ points uniformly distributed on $[a,b]$ ($M\lesssim 10^{5}$, chosen to keep the run time under control).
Generally, it is necessary to sample more densely near a singularity if there is one; this is important e.g.~for $f(x) = |x|$. 
We incorporate this need as follows: use AAA to find the support points $\{t_k\}$ (assume they are sorted), and take $M/n$ points between $[t_k,t_{k+1}]$. 
%The idea is that AAA tends to take more support points near the  oscillatory (or singular) parts of the function.  % deleted to save space

\section{A barycentric version of the differential correction algorithm}\label{sec:dc}

The DC algorithm, due to Cheney and Loeb~\cite{cheney1961two}, has the great advantage of guaranteed global convergence in theory~\cite{barrodale1972differential,dua1973further}, which applies whether the approximation domain $X$ is an interval $[a,b]$ or a finite set.  It can also be extended to multivariate approximation problems~\cite{hettich1990algorithm}. In practice, however, it may suffer greatly from rounding errors, and its speed is often disappointing on larger problems. As we shall now describe, we have found that the first of these difficulties can be largely eliminated by the use of barycentric representations with adaptively chosen support
points. The second problem of speed, however, remains, which is why ultimately we prefer the Remez algorithm for most problems.

%The differential correction (DC) algorithm, due to Cheney and Loeb~\cite{cheney1961two}, is sometimes considered to be more robust than the Remez algorithm, provided that numerical errors do not factor in~\cite[p.~166]{braess2012nonlinear}. This is due to its global convergence properties~\cite{barrodale1972differential,dua1973further} which hold when the approximation domain $X$ is either $[a,b]$, or a finite set. It can also be successfully extended to multivariate approximation problems~\cite{hettich1990algorithm}.

%In an implementation setting however, the Remez algorithm generally tends to be more practical for large $m,n$ when theoretical complications related to initialization and degeneracy are absent. It is therefore interesting to investigate if the use of a barycentric basis for the differential correction algorithm also provides significant numerical benefits. This section explores this idea to a certain extent.

\subsection{The barycentric formulation} For an effective implementation, $X$ needs to be a finite set (e.g.~obtained by discretizing $[a,b]$) 
to reduce each iteration to a linear programming (LP) problem. Considering the diagonal case $m=n$, a barycentric version of the DC algorithm can be defined recursively as follows. (We assume the support points are fixed to the values $t_0,\ldots,t_n$, which do not belong to $X$.) Given $r_k=N_k/D_k\in\mathcal{R}_{n,n}(X)$, choose the partial fraction decompositions $N$ and $D$ of~\eqref{eq:baryr} that minimize the expression
\begin{equation}\label{eq:diffcor}
\max_{x\in X}\left\lbrace\frac{\left|f(x)D(x)-N(x)\right|-\delta_k\left|D(x)\right|}{\left|D_k(x)\right|}\right\rbrace,
\end{equation}
subject to
\begin{equation}\label{eq:samesign}
\textnormal{sign}(\omega_t(x)D(x)) = \textnormal{sign}(\omega_t(y)D(y)),\qquad \forall x,y\in X, \quad x\neq y,
\end{equation}
and
\begin{equation}\label{eq:normalizeDC}
\max_{0\leq j\leq n}\left|\beta_j\right|\leq 1,
\end{equation}
where $\delta_k=\max_{x\in X}\left|f(x)-r_k(x)\right|$. 
%(the bound $n$ is arbitrary here, for example $1$ would work too. However, we have observed that taking it to be larger improves the robustness). 
If $r=N/D$ is not good enough, continue with $r_{k+1}=r$. 
By imposing~\eqref{eq:normalizeDC}, we can establish convergence using an argument analogous to~\cite[Theorem~2]{barrodale1972differential}. In the polynomial basis setting, we know that the rate of convergence will ultimately be at least quadratic if the best approximation is non-degenerate~\cite[Theorem~3]{barrodale1972differential}. 
%\rr{It is possible to extend the convergence proof to the barycentric version described above.}
Non-diagonal approximations can be computed by adding the appropriate null space constraints as described in Section~\ref{sec:nondiag}.

%A conceptually simpler approach that is also based on linear programming~\cite[Ch. 10.4]{powell1981approximation} is to take an initial guess $\delta>0$ for the minimax error and check if there exist barycentric coefficients $\left\lbrace\alpha_k\right\rbrace, \left\lbrace\beta_k\right\rbrace$ such that
%\begin{equation}\label{eq:feastest}
%	\begin{array}{c}
%	\left|f(x)D(x)-N(x)\right|\leq \delta\cdot\textnormal{sign}(\omega_t(x))D(x), \\
%\vspace{2mm}
	%\textnormal{sign}(\omega_t(x))D(x)>0, \\
%	\beta_0=1,
%	\end{array} \qquad x\in X.
%\end{equation}
%Then, by using a bisection procedure that adapts the value of $\delta$, $N/D$ will also converge to the minimax solution over $X$. Assuming $X$ is taken to be sufficiently dense inside $[a,b]$, we noticed numerical improvements for both approaches in comparison to using a more traditional monomial or Chebyshev basis.

\subsection{Choice of support points}\label{sec:dcsupport} Compared to the case of the barycentric Remez algorithm, changing the support points at each iteration of the DC algorithm makes it hard to impose a normalization condition similar to~\eqref{eq:normalizeDC} or do a convergence analysis of the method. 
We therefore fix $\{t_k\}$ throughout the execution. The strategy we have adopted is based on Section~\ref{sec:inicdf}: recursively construct type $(\ell,\ell)$ approximations with $\ell\leq n$. We take the set of support points of the $(\ell,\ell)$ problem based on a piecewise linear fit of the final reference points of the $(\ell-1,\ell-1)$ problem (similar to what is shown in Figure~\ref{fig:lowerdegree}). % and apply~\eqref{eq:supportchoose} to take the support points as the midpoints of every other pair of consecutive reference points $x_{2k}$ and $x_{2k+1},k=0,\ldots,\ell$. A similar procedure can be used for the bisection routine inspired by~\eqref{eq:feastest}, although modifying the support points here is not problematic for convergence.

\begin{table}[h!]
	\small
	\centering
	\caption{Best type $(16,16)$ approximations to four functions using the barycentric DC algorithm. $X$ consists of $20000$ equispaced points inside $[-1,1]$.}
	\label{table:dc}
	\begin{tabular}{lcc}
		\hline 
		$i$ & $f_i$ & $\|f_i-r^*\|_{X,\infty}$ \T\B \\ \hline \vspace{2mm}
&& \vspace{-4mm}\\ \vspace{2mm} % extra space
		1&  $\sum_{k=0}^{\infty}2^{-k}\cos(3^kx)$     &     $0.1377$          \\ \vspace{2mm}
		2&  $\min\left\lbrace \textnormal{sech}(3\sin(10x)),\sin(9x)\right\rbrace$     & $0.0610$              \\ \vspace{2mm}
		3&  $\sqrt{|x^3|}+|x+0.5|$     &   $1.2057\cdot 10^{-4}$            \\ 
		4&  $\left(\frac{1}{2}\textnormal{erf}\frac{x}{\sqrt{0.0002}}+\frac{3}{2}\right)e^{-x}$     &       $6.2045\cdot 10^{-6}$        \vspace{2mm}\\ \hline
	\end{tabular}
\end{table}

\begin{figure}[h!]
	\centering
	\includegraphics[width=\linewidth]{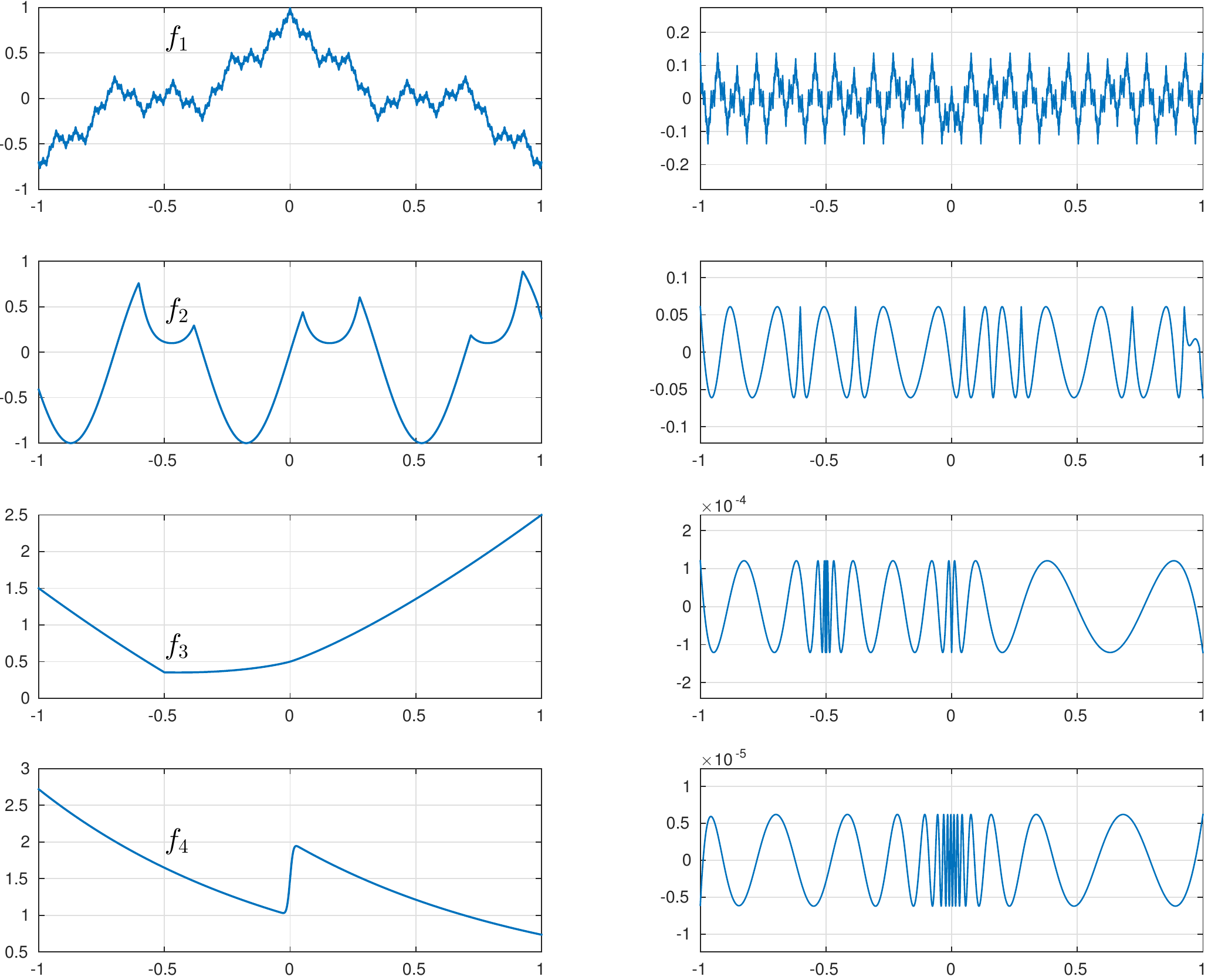}
	\caption{The functions of Table~\ref{table:dc} with error curves for best rational approximations computed by the barycentric DC algorithm.}
	\label{fig:dcErrors}
\end{figure}

\subsection{Experiments}We have implemented\footnote{The prototype code used is available at \href{https://github.com/sfilip/barycentricDC}{https://github.com/sfilip/barycentricDC}.} the barycentric DC algorithm in MATLAB using CVX~\cite{cvx} to specify the LP problems corresponding to~\eqref{eq:diffcor}--\eqref{eq:normalizeDC}, which are then solved using MOSEK's~\cite{mosek} state-of-the-art LP optimizers. The four examples in Table~\ref{table:dc} and Figure~\ref{fig:dcErrors}, for instance, demonstrate the effectiveness of the algorithm. For comparison, the sensitivity to the initial reference set prevented the convergence of our barycentric Remez implementation on \emph{all four of} these examples. Function $f_1$ is particularly interesting since it is a version of Weierstrass's classic example of a continuous but nowhere differentiable function.

Using a monomial or Chebyshev basis representation for the LP formulations quickly failed due to numerical errors, illustrating that the barycentric representation is crucial for the DC algorithm just as for the Remez algorithm.

We nevertheless echo the statement in the beginning of the section of the downsides of using the DC approach:
\begin{itemize}
\item Its overall cost. Producing the approximations in Figure~\ref{fig:dcErrors} took several minutes in MATLAB on a desktop machine for each example.
\item Numerical optimization tools for solving the corresponding LP problems break down at lower values of $m$ and $n$ than the ones we achieved with the barycentric Remez algorithm. We were usually able to go up to about type $(20,20)$.
% with our barycentric DC algorithm implementation.
\end{itemize}

\section{Minimax approximation in Chebfun}\label{sec:conclude}
\begin{figure}[h!]
	\centering
	\includegraphics[width=\linewidth,trim=0 0 35 0, clip]{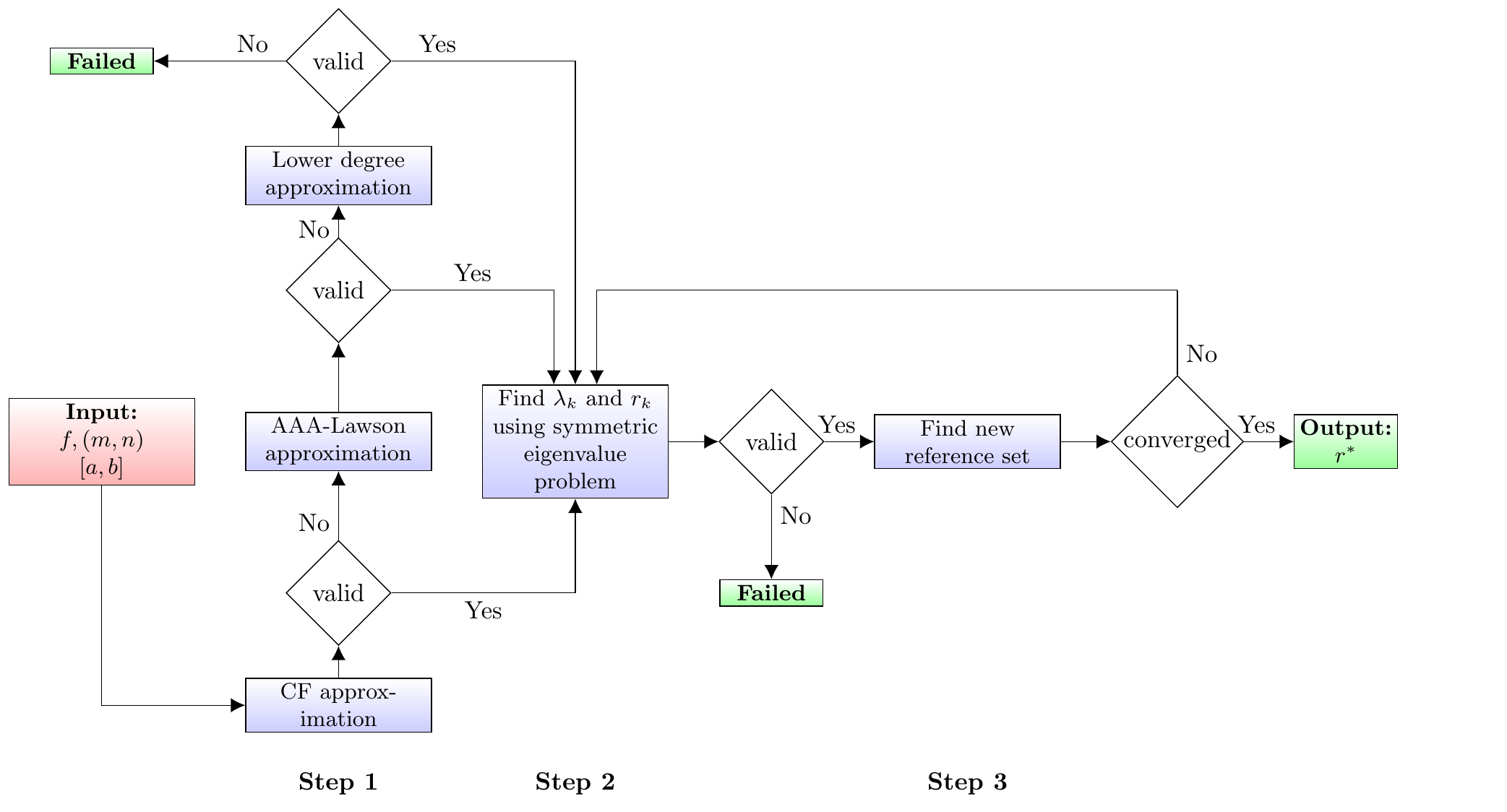}
	\caption{Flowchart summarizing the~\texttt{minimax} implementation of the rational Remez algorithm in the unweighted case. It follows the steps outlined at the start of Section~\ref{sec:remezbasics}. Step 1 consists of picking the initial reference set. This is done by applying in succession (if needed) the strategies discussed in Sections~\ref{sec:inicf},~\ref{subsubsec:AAAlaw} and~\ref{sec:inicdf}. Next up in Step 2 is computing the current approximant $r_k$ and alternation error $\lambda_k$. We do this by solving a symmetric eigenvalue problem~\eqref{eq:retbeta},~\eqref{eq:retbetamgn} or~\eqref{eq:retbetammn}, depending on $m=n$, $m>n$ or $m<n$. 
%derived in~\eqref{eq:retbeta} for the diagonal $m=n$ case and one of the analogous equations~\eqref{eq:retbetamgn} and~\eqref{eq:retbetammn} for the nondiagonal $m\neq n$ setting. 
We then pick, if possible, the eigenpair leading to a rational approximant with no poles in $[a,b]$ (see discussion around equation~\eqref{eq:valqbary}). 
%To determine the barycentric coefficients $\alpha$ and $\beta$ of $r_k$ we use $y=R\beta$ on the eigenvector candidate $y$ and~\eqref{eq:retalpha} for $\alpha$, with analogous equations derived in Section~\ref{sec:nondiag} for the nondiagonal setting. 
The next reference set is determined in Step 3 as explained in Section~\ref{sec:nextrefs}. If convergence is successful, the routine outputs a numerical approximant of $r^*$.}
	\label{fig:flowChart}
\end{figure}
We have presented many algorithmic details that have enabled the design of a fast and robust Remez implementation. 
In closing we remind readers that all this is available in Chebfun and readily explored in a few lines of code.  Download Chebfun version 5.7.0 or later from GitHub or \url{www.chebfun.org}, put it in your MATLAB path, and then try for example
\medskip
\begin{verbatim}
      [p,q,r] = minimax(@(x) abs(x),60,60);
      fplot(@(x) abs(x)-r(x),[-1 1])
\end{verbatim}
\medskip
In a few seconds a beautiful curve with 123 exponentially clustered equioscillation points will appear.
Figure~\ref{fig:flowChart} summarizes our algorithm in a flowchart.

%\subsection*{Acknowledgement}\rr{We thank the reviewers for their useful comments and suggestions, which helped improve the quality of the paper.}

\bibliographystyle{abbrv}
\bibliography{biblio}

\end{document}